\documentclass[11pt,oneside]{amsart}
\usepackage{amsmath,amssymb,graphicx,subfigure}
\graphicspath{{.}{./eps/}}
\newcommand{\comment}[1]{\marginpar{\sffamily{\tiny #1
\par}\normalfont}}

\renewcommand{\comment}[1]

\newcommand{\bi}{\begin{itemize}}
\newcommand{\ei}{\end{itemize}}
\newcommand{\be}{\begin{enumerate}}
\newcommand{\ee}{\end{enumerate}}
\newcommand{\bc}{\begin{center}}
\newcommand{\ec}{\end{center}}
\newcommand{\bt}{\begin{tabular}}
\newcommand{\et}{\end{tabular}}
\newcommand{\bcase}{\vspace{3mm}
\noindent \textit{Case }}

\newtheorem{thm}{Theorem}[section]
\newtheorem{lem}[thm]{Lemma}
\newtheorem{propn}[thm]{Proposition}
\newtheorem{cor}[thm]{Corollary}

\theoremstyle{definition}

\newcommand{\blackboard}[1]{\ensuremath{\mathbb{#1}}}
\newcommand{\script}[1]{\ensuremath{\mathcal{#1}}}

\newcommand{\N}{\blackboard{N}}
\newcommand{\Z}{\blackboard{Z}}

\newcommand{\R}{\blackboard{R}}
\newcommand{\C}{\script{C}}

\newcommand{\hnn}{HNN extension}

\newcommand{\ac}{almost convex}
\newcommand{\mac}{$MAC$}

\newcommand{\mpac}{$M'AC$}

\newcommand{\ip}{isoperimetric function}

\newcommand{\ts}{\sigma_t}
\newcommand{\ra}{\rightarrow}
\newcommand{\g}{G}
\newcommand{\ow}{\overline{w}}
\newcommand{\ou}{\overline{u}}

\newcommand{\ov}{\overline}
\newcommand{\dd}{\delta}

\newcommand{\Epstein}{MR93i:20036}

\newcommand{\Bridson}{MR2000g:20071}
\newcommand{\Cannon}{MR88a:20049}
\newcommand{\Thiel}{MR95e:20052}

\newcommand{\MShap}{MR99k:20075}

\newcommand{\Stallings}{MR28:2139}
\newcommand{\Kapovich}{MR1879521}
\newcommand{\Riley}{MR1931373}
\newcommand{\Poenaru}{MR93d:57032}
\newcommand{\Funar}{MR99k:20074}
\newcommand{\Miller}{Miller}

\newcommand{\Belkbux}{MR2139683}
\newcommand{\Groves}{MR98b:20044}
\newcommand{\Guba}{MR2207020}

\begin{document}

\title[Minimal Almost Convexity]
  {Minimal Almost Convexity}

\author[M.~Elder]{Murray Elder}
\address{School of Mathematics and Statistics\\
University of St. Andrews\\
North Haugh, St. Andrews\\
Fife, KY16 9SS, Scotland}
\email{murray@mcs.st-and.ac.uk}

\author[S.~Hermiller]{Susan Hermiller$\!\,^1$}
\address{Department of Mathematics\\
        University of Nebraska\\
         Lincoln NE 68588-0323, USA}
\email{smh@math.unl.edu}
\subjclass[2000]{20F65}
\keywords{(Minimally) almost convex, Baumslag-Solitar group, Stallings group}
\date{\today}

\begin{abstract}
In this article we show
that the Baumslag-Solitar group $BS(1,2)$ is minimally almost convex, 
or \mac.  We also
show that $BS(1,2)$ does not satisfy Po\'enaru's almost
convexity condition $P(2)$, and hence the condition $P(2)$
is strictly stronger than \mac.
Finally, we show that the groups $BS(1,q)$ for $q \ge 7$ and
Stallings' non-$FP_3$ group do not satisfy \mac.  As a
consequence, the condition \mac\  is not a commensurability
invariant.
\end{abstract}

\maketitle

\footnotetext[1]{Supported under NSF grant no.\ DMS-0071037}

\section{Introduction}

Let $G$ be a group with finite generating set $A$, 
let $\Gamma$ be the corresponding Cayley graph with the
path metric $d$, and let $S(r)$ and $B(r)$ denote the sphere 
and ball, respectively, of
radius $r$ in $\Gamma$.
The pair $(G,A)$ satisfies the almost
convexity condition
$AC_{f,r_0}$ 
for a function $f:\N \ra \R_+$ and a natural number 
$r_0 \in \N$ 
if for every natural number
$r \ge r_0$ and every pair of vertices $a,b \in S(r)$ with
$d(a,b) \le 2$, there
is a path inside $B(r)$ from $a$ to $b$ of length at most $f(r)$.
Note that every group satisfies the condition $AC_{f,1}$ for the
function $f(r)=2r$.  A group is {\it minimally almost convex}, or
\mac\ (called $K(2)$ in \cite{\Kapovich}), if the condition
$AC_{f,r_0}$ holds for the function $f(r)=2r-1$ and some number $r_0$;
that is, the least restriction possible is imposed on the function
$f$.  If the next least minimal restriction is imposed, i.e. if $G$ is
$AC_{f,r_0}$ with the function $f(r)=2r-2$, then the group is said to
be \mpac\ ($K'(2)$ in \cite{\Kapovich}).  Increasing the restriction
on the function $f$ further, the group satisfies Poenaru's $P(2)$
condition \cite{\Funar,\Poenaru} if $AC_{f,r_0}$ holds 
for a sublinear function
$f: \N \ra \R_+$, i.e. $f$ satisfies the property that for every
number $C>0$ the limit $\lim_{r \ra \infty}(r-Cf(r)) = \infty$.  All
of these definitions are generalizations of the original concept of
almost convexity given by Cannon in \cite{\Cannon}, in which the greatest
restriction is placed on the function $f$, namely that a group is {\it
almost convex} or $AC$ if there is a constant function $f(r) \equiv C$
for which the group satisfies the condition $AC_{C,1}$.
Results of \cite{\Cannon,\Kapovich,\Riley} show that
the condition \mac, and hence each of the other almost convexity
conditions, implies finite presentation of the group and solvability
of the word problem.

The successive strengthenings of the restrictions in
the definitions above give the implications
$AC \Rightarrow P(2) \Rightarrow$ \mpac\ $\Rightarrow$ \mac.
For this series of implications, a natural question to ask is, 
which of the
implications can be reversed?  
A natural family of groups to consider are the
Baumslag-Solitar groups 
$BS(1,q):=\langle a,t ~|~ tat^{-1}=a^q \rangle$ with $|q|>1$, which 
Miller and Shapiro \cite{\MShap} proved are
not almost convex with respect to any generating set.  

In the present paper, the structure
of geodesics in the Cayley graph of $BS(1,q)$ is 
analyzed in greater detail, in Section 2.
In Section 3, we use these results
to show that the group $BS(1,2)$
satisfies the property \mpac.  
In Section 4 we show that 
the group $BS(1,2)$ does not satisfy the $P(2)$ condition,
and hence the implication $P(2) \Rightarrow$ \mpac\  cannot
be reversed.

In section 4 we also show that the groups 
$BS(1,q)=\langle a,t ~|~ tat^{-1}=a^q \rangle$ for $q \ge 7$ are
not \mac.  Since the group $BS(1,8)$ is a finite index
subgroup of $BS(1,2)$, an immediate consequence of
this result is that both \mac\ and \mpac\ are not
commensurability invariants, and hence not quasi-isometry
invariants.  The related property $AC$ is also known to
vary under quasi-isometry; in particular, Thiel \cite{\Thiel}
has shown that $AC$ depends on the generating set.

Finally, in Section 5 we
consider Stallings' non-$FP_3$ group \cite{\Stallings}, which
was shown by the first author \cite{Ethesis, Elooppaper} 
not to be \ac\ with respect to two different finite generating sets.
In Theorem \ref{thm:stall}, we prove the stronger result that this group
also is not \mac, with respect to one of the generating sets.
Combining this with a result of Bridson \cite{\Bridson} that
this group has a quadratic \ip, we obtain 
an example of a group with quadratic \ip\  that is not \mac.
During the writing of this paper,
Belk and Bux \cite{\Belkbux} have shown another
such example; namely, they have shown that 
Thomson's group $F$, which also has a quadratic
\ip\ function \cite{\Guba}, does not satisfy \mac.

\section{Background on Baumslag-Solitar groups}
 
Let $G:=BS(1,q)=\langle a,t~|~tat^{-1}=a^q \rangle$ with generators
$A:=\{a,a^{-1},t,t^{-1}\}$ for any natural number $q>1$.
Let $\Gamma$ denote the corresponding
Cayley graph with path metric $d$, and let ${\C}$ denote
the corresponding Cayley 2-complex.

The complex ${\C}$ can be built from ``bricks''
homeomorphic to $[0,1] \times [0,1]$, with
both vertical sides labeled by a ``$t$'' upward, the top
horizontal side labeled by an ``$a$'' to the right, and
the bottom horizontal side split into $q$ edges
each labeled by an ``$a$'' to the right.  These bricks
can be stacked horizontally into ``strips''.
For each strip, $q$ other strips can
be attached at the top, and one on the bottom.
For any set of successive choices upward, then, the
strips of bricks can be stacked
vertically to fill the plane.  
The Cayley complex then is homeomorphic to the Cartesian
product of the real line with a regular tree $T$ of valence $q+1$;
see Figure \ref{fig:brick}.
\begin{figure}
\bt{cc}
\includegraphics[width=6cm]{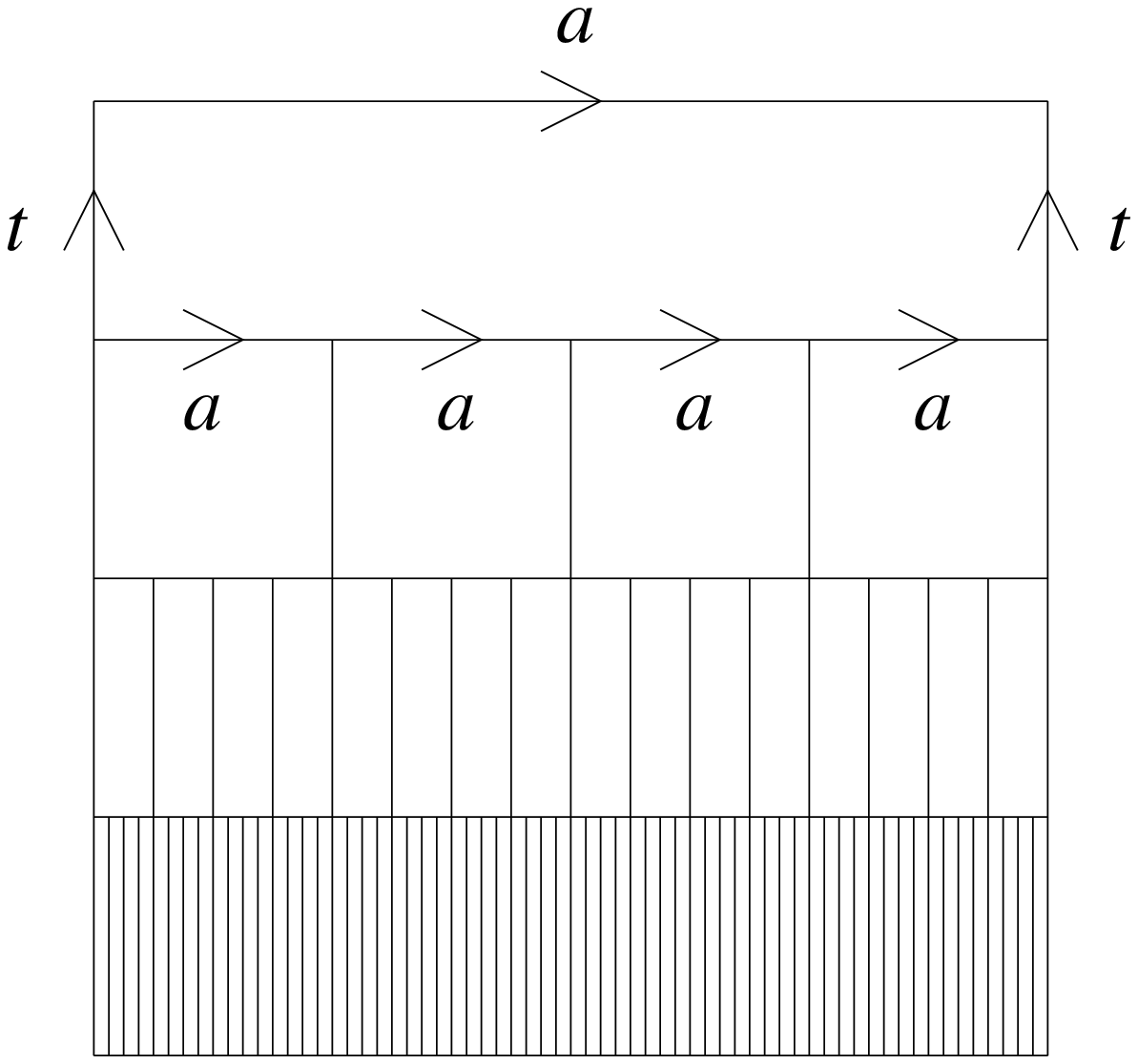}
&
\includegraphics[width=6cm, height=5.3cm]{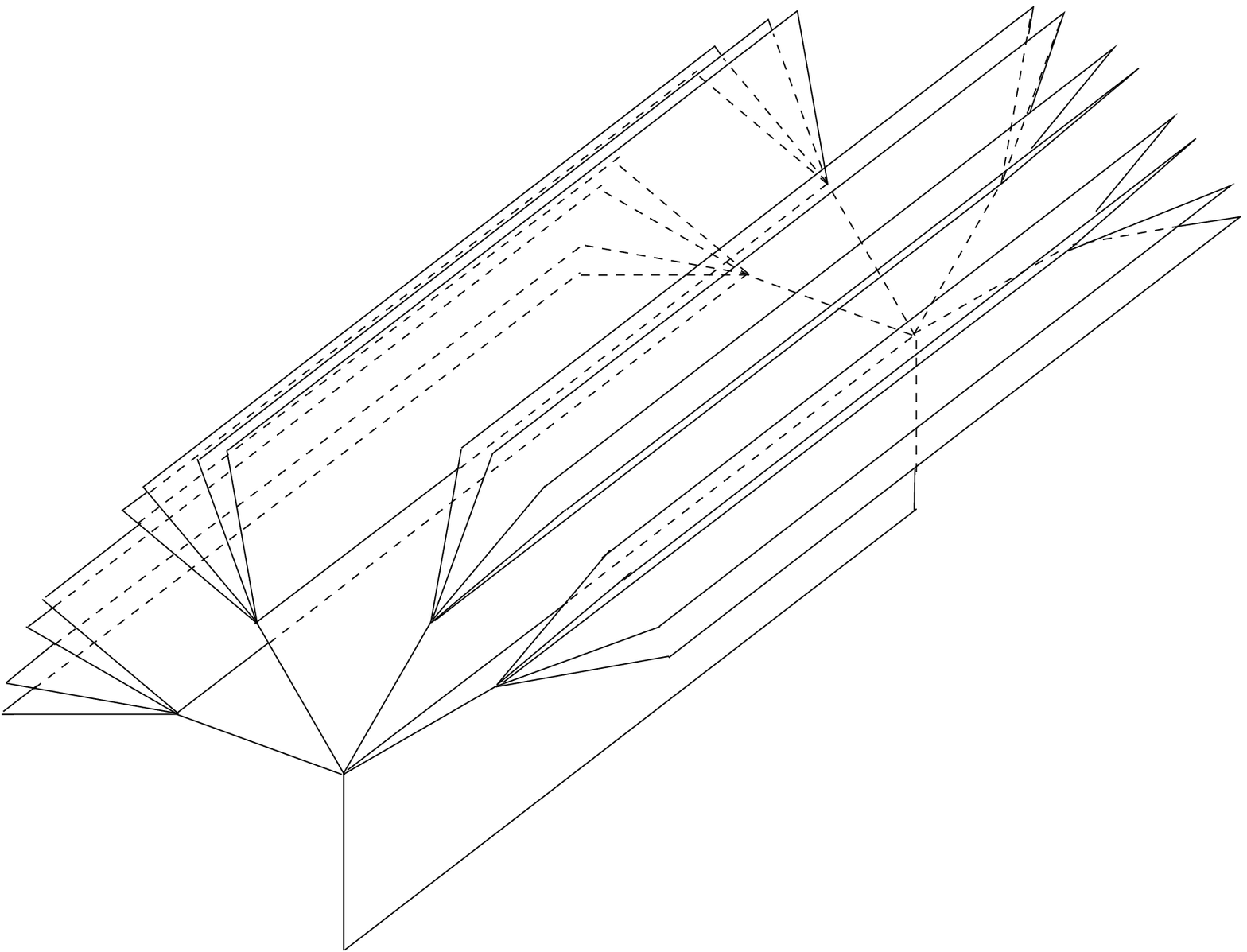} 
\et
\caption{A brick in a plane, and a side-on view of the  Cayley graph $\Gamma$ for $BS(1,4)$.}
\label{fig:brick}
\end{figure}
Let $\pi:\C \ra T$ be the horizontal projection map.
For an edge $e$ of $T$, $e$ inherits an upward direction from the
upward labels on the vertical edges of $\C$ that project
onto $e$.
More details can be found in \cite{\Epstein} (pages 154-160).

For any word $w \in A^*$, let
$\ow$ denote the image of $w$ in $BS(1,q)$.  For words $v,w \in
A^*$, denote $v=w$ if $v$ and $w$ are the same words in $A^*$, and
$v=_{\g} w$ if $\ov{v}=\ow$. Let 
$l(w)$ denote the word length of $w$ and let
$w(i)$ denote the 
prefix of the word $w$ containing $i$ letters.
Then $(w^{-1}(i))^{-1}$ is the suffix of $w$ of length $i$.
Define $\ts(v)$ to be the exponent sum of all occurrences of 
$t$ and $t^{-1}$ in $v$.  Note that 
relator $tat^{-1}a^{-q}$ in the presentation of $G$ satisfies
$\ts(tat^{-1}a^{-q})=0$; hence whenever
$v =_{\g} w$, then $\ts(v)=\ts(w)$.

The following Lemma is well-known; a proof can be
found in \cite{\Miller}.

\begin{lem}[Commutation]\label{lem:commute}
If $v,w \in A^*$ and $\ts(v)=0=\ts(w)$,
then $vw=_{\g}wv$.
\end{lem}

Let $E$ denote the set of words in $\{a,a^{-1}\}^*$, $P$ the words in
$\{a,a^{-1},t\}^*$ containing at least one $t$ letter, and $N$ the words
in $\{a,a^{-1},t^{-1}\}^*$ containing at least one $t^{-1}$ letter.  
A word $w=w_1w_2$ with $w_1 \in N$ and $w_2 \in P$, then, will be
referred to as a word in $NP$.  
Finally, let $X$ denote the subset of
the words in $PN$ with $t$-exponent sum equal to 0.
Letters in parentheses denote subwords that may or
may not be present; for
example, $P(X):=P \cup PX$.
The following statement is proved in \cite{\Groves}.

\begin{lem}[Classes of geodesics]  \label{lem:class}
A word $w \in A^*$ that is a geodesic in $\Gamma$ must fall into one
of four classes:
\begin{enumerate}
\item E or X,
\item N or XN,
\item P or PX,
\item NP, or NPX with $\ts(w) \ge 0$, or XNP with $\ts(w) \le 0$.
\end{enumerate}
\end{lem}

Analyzing the geodesics more carefully, we find
a normal form for geodesics in the following Proposition.

\begin{propn}[Normal form]  \label{propn:normalform}
If $w \in A^*$ is a geodesic in $G=BS(1,q)$,
then there is another geodesic $\hat w \in A^*$
with $\ov{\hat w}=\ov{w}$ such that for $w$ in each class, $\hat w$
has the following form (respectively):
\begin{enumerate}

\item $\hat w=a^i$ for 
$|i|\leq C_q$ where
$C_q:=\lfloor \frac{q}{2}+1 \rfloor$ if $q>2$ and $C_2:=3$, 
or $\hat w =w_0 \in X$,

\item $\hat w=w_0t^{-1}a^{m_1} \cdots t^{-1}a^{m_e}$ with 
$|m_j| \le \lfloor \frac{q}{2} \rfloor$ for all $j$, $e \ge 1$, 
and either $w_0=a^i$ for $|i| \le C_q$ or
$w_0 \in X$,

\item $\hat w=a^{n_0}t \cdots a^{n_{f-1}}tw_0$ with 
$|n_j| \le \lfloor \frac{q}{2} \rfloor$ for
all $j$, $f \ge 1$, 
and either $w_0=a^i$ for $|i| \le C_q$ or $w_0 \in X$,

\item Either $\hat w=t^{-e}a^{m_f}ta^{m_{f-1}} \cdots a^{m_{1}}tw_0$
  with $1 \le e\leq f$, or 
  $\hat w=w_0t^{-1}a^{m_1} \cdots t^{-1}a^{m_e}t^f$
  with $1 \le f \leq e$, such that 
  $|m_j| \le \lfloor \frac{q}{2} \rfloor$ for all $j$,
  and either $w_0=a^i$ for $|i| \le C_q$ or $w_0 \in X$.  Note that if
  $\ts(w)=0$ then $e=f$ and either expression is valid.
\end{enumerate}
In every class the word $w_0 \in X$ can be chosen to be 
either of the form 
$w_0=t^{h}a^{s}t^{-1}a^{k_{h-1}} \cdots a^{k_{1}}t^{-1}a^{k_0}$
or
$w_0=a^{k_0}ta^{k_{1}} \cdots ta^{k_{h-1}}ta^st^{-h}$ 
with $|k_j| \le \lfloor \frac{q}{2} \rfloor$ for
all $j$, $1 \le |s| \le q-1$ if $q>2$,
$2 \le |s| \le 3$ if $q=2$, and $h \ge 1$.
\end{propn}

\begin{proof} Note that
the natural number 
$q=\lfloor \frac{q}{2} +1 \rfloor+\lceil \frac{q}{2}-1 \rceil$.

For a geodesic $w$ in class (1), if $w \in E$, then
$w=a^i$ for some $i$.  
If $q=2$, then $a^{\pm 6}=ta^{\pm 3}t^{-1}$ so
$|i| \le 6$, and the words $a^{\pm(4+k)}$ have normal form
$ta^{\pm 2}t^{-1}a^{\pm k} \in X$ for $k=0$ and $k=1$.
If $q>2$, then the relation 
$tat^{-1}=_{\g} a^q$ can be reformulated as
$a^{\pm \lfloor \frac{q}{2}+1 \rfloor+1}=_{\g}
 ta^{\pm 1}t^{-1}a^{\mp(\lceil \frac{q}{2}-1 \rceil-1)}$.
If $q$ is even, then $a^{\pm \frac{q}{2}+2}$
is not geodesic, so $|i| \le \lfloor \frac{q}{2} +1 \rfloor$.
On the other hand, if $q$ is odd,
then  
$a^{\pm {q+1 \over 2}+2}=_{\g}
 ta^{\pm 1}t^{-1}a^{\mp({q-1 \over 2}-2)}$ so
$a^{\pm {q+1 \over 2}+2}$ is not geodesic; hence
$|i| \le \lfloor \frac{q}{2} +1 \rfloor+1$, and
the words
$a^{\pm \lfloor \frac{q}{2} +1 \rfloor+1}$ have a normal form in $X$.

Next suppose that $w$ is a geodesic in class (2).  
Then $w \in (X)N$, so 
$w=w_0't^{-1}a^{l_1}t^{-1}a^{l_2} \cdots t^{-1}a^{l_e}$
for some word $w_0'$ in class (1), $e \ge 1$,
and integers $l_i$.  
Again we reformulate the defining relation of $G$, in this case 
to $t^{-1}a^{\pm \lfloor \frac{q}{2} +1 \rfloor}=_{\g}
 a^{\pm 1}t^{-1}a^{\mp(\lceil \frac{q}{2}-1 \rceil)}$.
If $q$ is odd, then we may (repeatedly) replace any occurrence of 
$t^{-1}a^{\pm \lfloor \frac{q}{2} +1 \rfloor}$ by
$a^{\pm 1}t^{-1}a^{\mp(\lceil \frac{q}{2}-1 \rceil)}$.
If $q$ is even, 
$t^{-1}a^{\pm \lfloor \frac{q}{2} +1 \rfloor}$ is not geodesic, so
$|l_j| \le \lfloor \frac{q}{2} \rfloor$ for all $j$ and 
replacements are not needed.
In both cases, then, we obtain a geodesic word 
of the form 
$w_0''t^{-1}a^{m_1}t^{-1}a^{m_2} \cdots t^{-1}a^{m_e}$ with 
each $|m_j| \le \lfloor \frac{q}{2} \rfloor$ and $w_0''$
in class (1);
to form the normal form $\hat w$, then, replace $w_0''$ by
its normal form.

The proof of the normal form for geodesics in class (3)
is very similar, using the relation
$a^{\pm \lfloor \frac{q}{2} +1 \rfloor}t=_{\g}
 a^{\mp(\lceil \frac{q}{2}-1 \rceil)}ta^{\pm 1}$.

Suppose next that $w$ is a geodesic in class (4) 
with $\ts(w) \ge 0$.  Then
$w=t^{-1}a^{k_1} \cdots a^{k_{e-1}}t^{-1}a^{l_f}t 
a^{l_{f-1}} \cdots a^{l_{1}}t w_0'$
with $w_0'$ in class (1), $1 \le e < f$, and
each $k_j,l_i \in Z$.  First use 
Lemma \ref{lem:commute} to replace
$w$ by the geodesic word 
$$
t^{-e} a^{l_f}ta^{k_{e-1}+l_{f-1}} \cdots ta^{k_{1}+l_{f-e+1}}
ta^{l_{f-e}}\tilde w_0ta^{l_{f-e-1}} \cdots  a^{l_{1}}t w_0'.
$$ 
To complete construction
of the normal form $\hat w$ from this word, replace the
subword $a^{l_f}t \cdots  a^{l_{1}}t w_0'$ by its normal
form from class (3).

The constructions for the normal forms of geodesics $w$ in class
(4) with $\ts(w) \le 0$, and of geodesics $w_0 \in X$, are
analogous.
\end{proof}


\section{The group BS(1,2) satisfies \mpac}

Let $G:=BS(1,2)=\langle a,t~|~tat^{-1}=a^2 \rangle$ with generators
$A:=\{a,a^{-1},t,t^{-1}\}$.  
In this section we prove, in Theorem \ref{thm:BS12mac}, that
this group is \mpac.  
We begin with a further analysis of the geodesics in $G$, via
several lemmas which are utilized in many of the cases in the proof
of Theorem \ref{thm:BS12mac}.

\begin{lem}[Large geodesic]\label{lem:rlarge}
If $w$ is a geodesic of length $r>200$ 
in one of the classes (1), (2), or (3) of 
Proposition \ref{propn:normalform} and $|\ts(w)| \le 2$, 
then $w$ is in either $X$, $XN$, or $PX$, respectively.
Moreover, the $X$ subword of $w$ must have the form
$w_1w_2$ with $w_1 \in P$ and  $w_2 \in N$ such that
$\ts(w_1)=-\ts(w_2)>10$.
\end{lem}
\begin{proof}
Suppose that $w$ is a geodesic in either $E$, $N$, or $P$ of
length $r>200$, and $|\ts(w)| \le 2$.  Then $w$ contains
at most two occurrences of the letters $t$ and $t^{-1}$.
As mentioned in the proof of Proposition \ref{propn:normalform},
$a^{\pm 6}=ta^{\pm 3}t^{-1}$ so $a^j$ is not geodesic
for $|j| \ge 6$.  Hence $w$ contains at most 15 occurrences
of the letters $a$ and $a^{-1}$ interspersed among the
$t^{\pm 1}$ letters.  Then $l(w) \le 17$, giving a contradiction.
\end{proof}

Given a word $w_0 \in X$, there is a natural number $k \in \N$ with
$w_0 =_{\g} a^k$; denote $\tilde w_0:= a^k$.
If $w$ is a geodesic word in $E \cup N \cup P \cup NP$, then
let $\tilde w := w$.  Combining these,
for any geodesic word $w=w_0w_1$ (or $w=w_1w_0$) with
$w_0 \in X$ and $w_1 \in N \cup P \cup NP$, define 
$\tilde w:=\tilde w_0 w_1 = a^kw_1$
(or $\tilde w:=w_1 \tilde w_0=w_1a^k$, respectively).  
Then $\tilde w \in N \cup P
\cup NP$, and the subword $w_1$ is geodesic.

\begin{lem}\label{lem:notgeod}
If $w$ is a word in $NP$, $NPX$ or $XNP$,
and $\tilde w$
contains a subword of the form $t^{-1}a^{2i}t$ with
$i\in \mathbb Z$, then $w$ is not geodesic.
\end{lem}

\begin{proof}
The word $w$ can be written as $w=w_0w_1w_2$ with
$w_1 \in NP$ and each of $w_0$ and $w_2$ either in
$X$ or $E$.  
Since $\tilde w=\tilde w_0  w_1 \tilde w_2$ contains
the subword
$t^{-1}a^{2i}t \in NP$, 
the word $t^{-1}a^{2i}t$ must be a subword of $w_1$, and
hence also of $w$.
Since $t^{-1}a^{2i}t =_{\g} a^i$, this subword is
not geodesic, and hence $w$ also is not geodesic.
\end{proof}

\begin{lem}\label{lem:britton}
If $w$ is any word in $NP$ or $NPN$ and $w=_{\g}1$,
then $w$ must contain a subword of the form 
$t^{-1}a^{2i}t$ for some $i\in \mathbb Z$. 
\end{lem}

\begin{proof}
Since $G=BS(1,2)$ is an HNN extension, 
Britton's Lemma states that if $w \in NP(N)$ and $w=_{\g}1$, 
then $w$ must contain a
subword of the form $ta^it^{-1}$ or $t^{-1}a^{2i}t$ for some $i\in
\mathbb Z$. If $w \in NP$ then $w$ must contain the second
type of subword. 

If $w \in NPN$, then $w=w_1w_2w_3$ with $w_1,w_3 \in N$ and
$w_2 \in P$.  Since $\ts(w_1)<0$ and 
$0=\ts(1)=\ts(w_1)+\ts(w_2)+\ts(w_3)$, 
$\ts(w_2) > \ts(w_3)$ and the word $w_2w_3 \in PX$.
Then $w_2w_3=w_4w_5$ with $w_4 \in P$ and $w_5 \in X$,
and the word $w_1 w_4 \tilde{w_5} \in NP$ with
$w_1 w_4 \tilde{w_5} =_{\g} w =_{\g} 1$.  Then Britton's Lemma 
applies again to show that the prefix $w_1w_4$ of $w$
must contain a subword of the form 
$t^{-1}a^{2i}t$ for $i\in \mathbb Z$.
%
\end{proof}

\comment{Prove lemma \ref{lem:twonp} using geometry instead?}

\begin{lem}\label{lem:twonp}
If $w$ and $u$ are geodesics, $w \in NP \cup XNP \cup NPX$,
$\ts(w) \le \ts(u)$, 
and $1 \le d(\ow, \ou) \le 2$, then
$u \in NP \cup XNP \cup NPX$ and for some
$w_1,u_1 \in N$ and $w_2,u_2 \in P$ with
$\ts(w_1)=\ts(u_1)$,
$\tilde w =w_1w_2$ and 
$\tilde u = u_1u_2$.
\end{lem}

\begin{proof}
The definition of $\tilde w$ shows that we can write
$\tilde w =w_1w_2$ with $w_1 \in N$ and $w_2 \in P$.
Let $\gamma$ label a path of length 1 or 2 from 
$\ow$ to $\ou$; since $\ts(w) \le \ts(u)$,
then $\gamma \in E \cup P$.
Proposition \ref{propn:normalform} says that 
$\tilde u \in E \cup P \cup N \cup NP$.
Since $w$ is a geodesic, 
Lemma \ref{lem:notgeod} implies that 
$\tilde w$ cannot contain a
subword of the form $t^{-1}a^{2i}t$ for any integer $i$.
Then Lemma \ref{lem:britton} says that the word
$\tilde w \gamma \tilde u^{-1}$, which represents 
the trivial element 1 in $\g$, cannot be in
$NP(N)$.  Therefore
$\tilde u \not\in E \cup P \cup N$, so $\tilde u \in NP$.
Hence $u \in NP \cup XNP \cup NPX$.

We can now write $\tilde u=u_1u_2$ with $u_1 \in N$ and
$u_2 \in P$.  The word
$\tilde u^{-1} \tilde w \gamma=u_2^{-1}u_1^{-1}w_1w_2\gamma$ is
another representative of 1.
Repeatedly reducing subwords $ta^jt^{-1}$ to $t^{2j}$ in
the subword $u_1^{-1}w_1 \in PN$ results in 
a word $\widetilde{u_1^{-1}w_1} \in E \cup P \cup N$.
Then $1=_{\g} u_2^{-1}\widetilde{u_1^{-1}w_1}w_2\gamma \in NP(N)$,
so this word must contain a 
subword of the form $t^{-1}a^{2i}t$ for some integer $i$.
Since $w$ and $u$ are geodesics, $w_1w_2$ and $u_2^{-1}u_1^{-1}$ 
cannot contain such a subword.  Therefore we must
have $\widetilde{u_1^{-1}w_1} \in E$.  Hence
$\ts(w_1)=\ts(u_1)$.
\end{proof}

We split the proof of Theorem \ref{thm:BS12mac} into 10 cases, depending on the classes from Proposition
\ref{propn:normalform} to which the two geodesics $w$ and $u$ belong.
In overview,
we begin by showing that the first three cases cannot occur;
that is, for a pair of length $r$
geodesics $w$ and $u$ in the respective classes
in these three cases, it is not possible for $d(\ow,\ou)$ to be
less than three.
In cases 4-6, we show that a path $\dd$ can be found that
travels from $\ow$ along the path $w^{-1}$ to within a distance 2 of
the identity vertex, and, after possibly
traversing an intermediate edge, $\dd$ then travels 
along a suffix of $u$ to $\ou$.
In case 7 we show that the path $\dd$ can be chosen to
have length at most six, traveling around at most two bricks
in the Cayley complex.  In case 8 there are subcases in which
each
of the two descriptions above occur, as well as a subcase in which
the path $\dd$ initially follows the inverse of a suffix of $w$ from $\ow$,
then travels along a path that ``fellow-travels'' this initial
subpath, and then repeats this procedure by traversing
a fellow-traveler of a suffix of $u$, and then traveling
along the suffix itself to $\ou$.  In cases 9 and 10,
the paths $\dd$ constructed in each of the subcases follow
one of these three patterns.

\begin{thm}\label{thm:BS12mac}
The group $G=BS(1,2)=\langle a,t ~|~ tat^{-1}=a^2\rangle$ is 
\mpac\  with respect to the 
generating set $A=\{a,a^{-1},t,t^{-1}\}$.  In particular,
if $w$ and $u$ are geodesics of length $r>200$ with 
$1 \le d(\ow,\ou) \le 2$, then there is a path 
$\delta$ inside $B(r)$ from
$\ow$ to $\ou$ of length at most $2r-2$.
\end{thm}

\begin{proof}
Suppose that $w$ and $u$ are geodesics of length $r>200$ with $1 \le
d(\ow,\ou) \le 2$.  Using Proposition \ref{propn:normalform}, by
replacing $w$ and $u$ by 
$\hat w$ and $\hat u$ respectively, we may assume that
each of $w$ and $u$ are in one of the normal forms listed in that
Proposition.  
Using Lemma \ref{lem:rlarge}, we may assume that
neither $w$ nor $u$ is in $E$.

Let $\gamma$ be the word labeling a geodesic path of length at most
$2$ from $\ow$ to $\ou$, so that $w \gamma u^{-1}=_{\g} 1$.  
Since $d(\ow,\ou) \ge 1$,
$$ \gamma \in \{a^{\pm 1},t^{\pm 1}, a^{\pm 2}, at^{\pm
1},a^{-1}t^{\pm 1}, ta^{\pm 1}, t^{-1}a^{\pm 1}, t^{\pm 2} \}.$$ 
Then $\gamma$ is in one of the sets $E$, $P$, or $N$.

We divide the argument into ten cases, depending on the 
class of the normal forms 
$w$ and $u$ from Proposition \ref{propn:normalform}, 
which we summarize in the following table.

\begin{table}[ht!]
\[
\bt{|l|l|l|l|l|l|l|}
\hline
Case & Class of $w$ & Class of $u$ && Case & Class of $w$ & Class of $u$ \\
\hline
Case 1:  & (4) & (1) && Case 6:  & (3) & (3) \\
\hline
Case 2:  & (4) & (3)   && Case 7:  & (2) & (2) \\
\hline
Case 3:  & (2) & (3) && Case 8:  & (1) & (3) \\
\hline
Case 4:  &  (1) & (1) &&  Case 9:  & (2) & (4) \\
\hline
Case 5:  & (1) & (2) && Case 10:  & (4) & (4) \\
\hline
\et
\]
\end{table}

This
table represents a complete list of the cases to be checked.
In particular, if $w$ is in class (2) and $u$ in class (1), 
then the inverse
of the path in Case 5 will provide the necessary path $\dd$,
and similarly for the remainder of the cases.

\bcase{\it 1: If $w$ is in class $(4)$ and $u$ is in class $(1)$:}  
Then $w$ is in either $NP$, $NPX$ or $XNP$, 
and $u \in X$.  
Since $\tilde w \in NP$, $\tilde u \in E$, and
the path $\gamma$ is either in  $E$, $N$ or $P$, then 
$1 =_{\g} w\gamma u^{-1} =_{\g} \tilde w \gamma \tilde u^{-1} \in NP(N)$
(that is, replacing the $X$ subwords of $w$ and $u$
by powers of $a$).  
By Lemma \ref{lem:britton}, 
$\tilde w \gamma \tilde u^{-1}$ 
contains a subword of
the form $t^{-1}a^{2s}t \in NP$, 
which therefore must occur within $\tilde{w}$.
Then Lemma \ref{lem:notgeod} says that
$w$ is not a geodesic, which is a
contradiction. Hence Case 1 cannot hold.

\bcase{\it 2: If $w$ is in class $(4)$ and $u$ is in class $(3)$:}  
Then $w$ is in either $NP$, $XNP$ or $NPX$, 
and $u \in P(X)$.  In this case
$1 =_{\g} \tilde w \gamma \tilde u^{-1} \in NPN$, and
the same proof as in Case 1 shows that Case 2
cannot occur.

\bcase{\it 3: If $w$ is in class $(2)$ and $u$ is in class $(3)$:}  
Then $w \in (X)N$ and $u\in P(X)$. 
Since $\ts(w)<0$ and $\ts(u)>0$ then we must
have $\ts(w)=-1$, $\ts(u)=1$, and $\gamma=t^2$.  
Lemma \ref{lem:rlarge} says that $w \in XN$, and $u \in PX$.
Since $w$ is in normal form, $w=\hat w=w_0t^{-1}a^i$ 
with $|i|\leq 1$ and $w_0 \in X$, and similarly $u=a^jtu_0$
with $|j| \leq 1$ and $u_0 \in X$. 
Then $1 =_{\g} 
\tilde w\gamma \tilde u^{-1} =_{\g} 
\tilde w_0 t^{-1}a^it^2 \tilde u_0^{-1} t^{-1}a^{-j} \in NPN$.
Lemma \ref{lem:britton} then says that
$w_0 t^{-1}a^it^2 \tilde u_0^{-1} t^{-1}a^{-j}$ contains
a subword of the form $t^{-1}a^{2s}t$ for some $s \in \Z$, 
so $i$ must be a multiple of $2$, and hence $i=0$.
Using the last part of Proposition \ref{propn:normalform},
we can further write the normal form for 
$w_0 \in X$ as $w_0=w_1t^{-1}$, so $w=w_it^{-2}$.
Then $u =_{\g} w \gamma =_{\g} w(r-2)$, contradicting
the hypothesis that $u$ is a geodesic word of length $r$.
Thus Case 3 does not hold.

\bcase{\it 4: If both $w$ and $u$ are in class $(1)$:}  
Then 
$w$ and $u$ are both in $X$.  From Proposition \ref{propn:normalform}
the normal forms $w=\hat w$ and $u=\hat u$ can be chosen of the form
$\hat w=t^hw_1$ and $\hat u=t^iu_1$ with $h,i >0$ and $w_1,u_1 \in N$.
Then $w$ and $u$ have a common prefix
$t=w(1)=u(1)$, and the path $\delta:=w_1^{-1}t^{-(h-1)}t^{i-1}u_1$
from $\ow$ through $\ov{w(1)}$ to $\ou$ has length
$2r-2$ and stays inside $B(r)$.

\bcase{\it 5: If $w$ is in class $(1)$ and $u$ is in class $(2)$:}
Then $w \in X$ and $u \in (X)N$.
In this case $\ts(w)=0$, 
$\ts(\gamma)=\ts(w^{-1}u)=\ts(w)+ \ts(u)=\ts(u)$, 
and $\ts(u)<0$, 
so $\ts(u)$ is either
$-1$ or $-2$.  
The hypothesis that $r>200$ and Lemma \ref{lem:rlarge} 
imply that
$u \in XN$.
Then both of the normal forms
$\hat w$ and $\hat u$ can be chosen to begin with $t$,
and the same proof as in Case 4 gives the path $\delta$.

\bcase{\it 6: If both $w$ and $u$ are in class $(3)$:}  In this case both
$w$ and $u$ are in $P(X)$.  Without
loss of generality assume that $\ts(w) \le \ts(u)$, so
$\ts(\gamma) \ge 0$ and $\gamma \in E \cup P$.
Since both $w$ and $u$ are in normal form,
$w=a^itw_1w_0$ and $u= a^jtu_1u_0$
with $w_1, u_1 \in P \cup E$; $w_0,u_0 \in X \cup E$; and
$|i|,|j|\leq 1$.  Then
$1 =_{\g} u^{-1}w\gamma =_{\g}
\tilde u_0^{-1}u_1^{-1}t^{-1}a^{i-j}tw_1\tilde w_0\gamma \in NP$.
By Lemma \ref{lem:britton},
$u_0^{-1}u_1^{-1}t^{-1}a^{i-j}tw_1\tilde w_0\gamma$ 
has a subword of the form $t^{-1}a^{2s}t$, so $i-j$ is a
multiple of $2$ and hence either $i=j$ with $0 \le |i| \le 1$
or $i=-j$ with $|i|=1$.

If $i=j$ then $w$ and $u$ have a common
prefix $a^it=w(1+|i|)=u(1+|i|)$.  The path
$\dd:=w_0^{-1}w_1^{-1}u_1u_0$ from $\ow$ follows the
suffix $w_1w_0$ of $w$ backward to $\ov{w(1+|i|)}$ and
then follows the suffix $u_1u_0$ of $u$ to $\ou$, remaining
in $B(r)$.

If $i=-j$ with $|i|=1$, define the path 
$$
\delta:=w_0^{-1}w_1^{-1}a^{-i}u_1u_0
=_{\g} w_0^{-1}w_1^{-1}t^{-1}a^{-i}a^{-i}tu_1u_0=w^{-1}u=_{\g}\gamma.
$$
Then $\delta$ labels a path 
of length $2r-3$, traveling along $w^{-1}$ from $\ow$ to 
$\ov{w\dd(r-2)}=\ov{w(2)}$, then along a single edge to 
$\ov{w\delta(r-1)}=\ov{u(2)}$, and finally along a suffix
of $u$ to $\ou$, thus remaining in $B(r)$.
(See Figure \ref{fig:case6}.)
\begin{figure}
\bt{ccc}
\includegraphics[height=4.7cm]{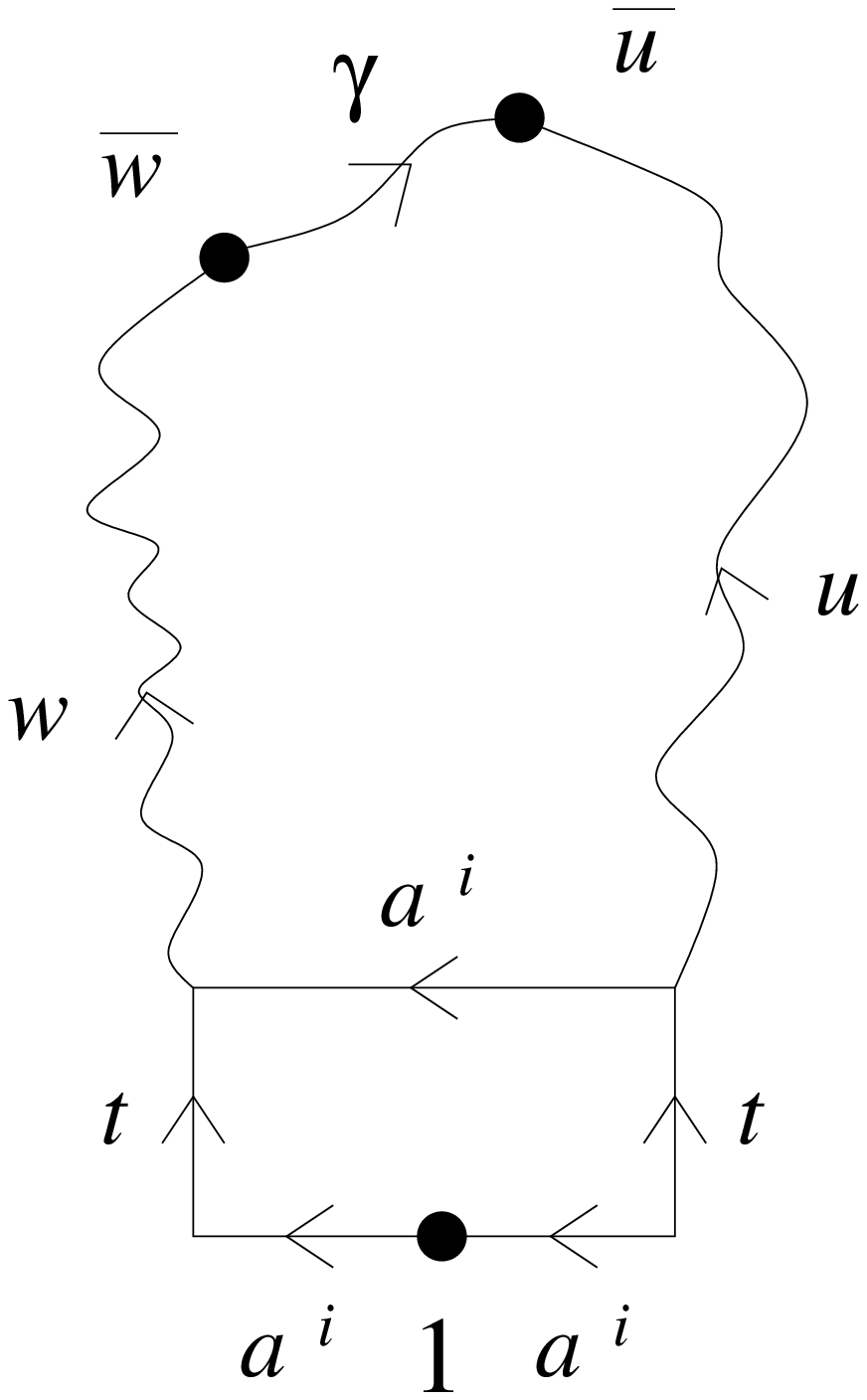} 
& \hspace{5mm} &
 \includegraphics[height=4.7cm]{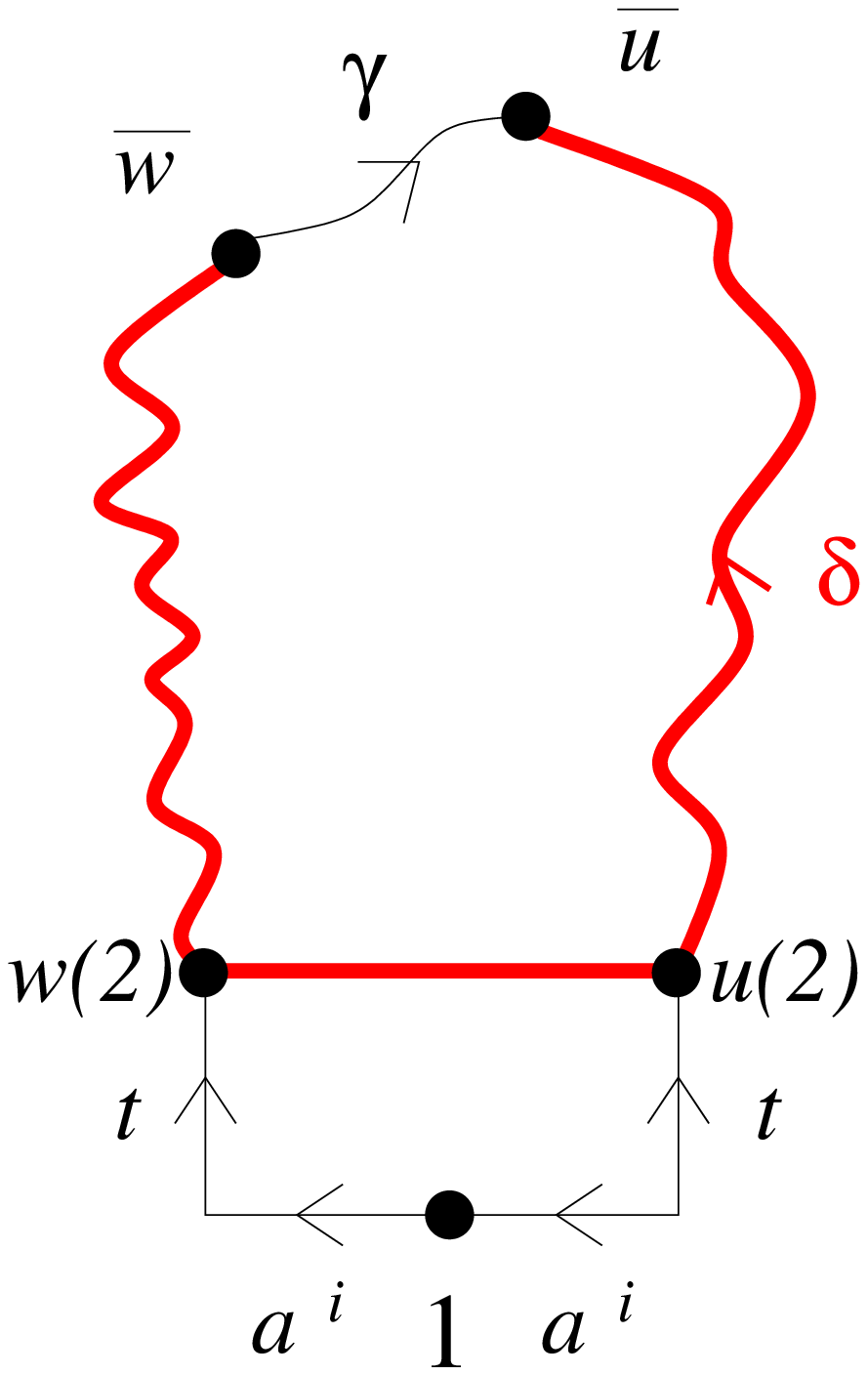} \et
\caption{Case 6:  $w=a^itw_1w_0$, $u=a^{-i}tu_1u_0$.}
\label{fig:case6}
\end{figure}

\bcase{\it 7: If both $w$ and $u$ are in class $(2)$:}  In this case, both
$w$ and $u$ are in $(X)N$.
We can assume without loss of
 generality that $\ts(w) \leq \ts(u)$, so again
$\ts(\gamma) \geq 0$ and $\gamma \in E \cup P$.

From 
Proposition \ref{propn:normalform} we have
$w=w_0w_1t^{-1}a^i$ and $u=u_0u_1t^{-1}a^j$ with 
$w_0,u_0 \in X \cup E$;
$w_1,u_1 \in N \cup E$; and $|i|,|j|\leq 1$.  
Thus 
$1 =_{\g} 
\tilde w \gamma \tilde u^{-1}= 
\tilde w_0w_1t^{-1}a^i \gamma a^{-j}tu_1^{-1}\tilde u_0^{-1} \in NP$. 
By Lemma \ref{lem:britton} the latter contains a subword of the form
$t^{-1}a^{2s}t$, and so $t^{-1}a^i\gamma a^{-j}t$ must contain this
subword. 

Since $\gamma \in E \cup P$, then 
$\gamma \in \{t, t^2, ta^{\pm 1}, a^{\pm 1}t, a^{\pm 1}, a^{\pm 2}\}$,
and we
may divide the argument into four subcases.

\bcase{\it 7.1: If $\gamma \in \{t,a^{\pm 1}t\}$:} 
Then $t^{-1}a^{2s}t$ must be a subword of
$t^{-1}a^i\gamma$.  
If $\gamma=t$, then since $|i| \le 1$ we have $i=0$
and
$w=w_0w_1t^{-1}$, so
$u =_G w \gamma =_{\g} w(r-1)$.  
If $\gamma=a^{\pm 1}t$, then
$|i|=1$, and 
$\gamma=a^{\pm i}t$.  
If $\gamma=a^it$, then 
$u =_{\g} w\gamma =w(r-2)t^{-1}a^ia^it
=_{\g} w(r-2)a_i$.  Finally, if
$\gamma=a^{-i}t$, then 
$u =_{\g} w\gamma =w(r-2)t^{-1}a^ia^{-i}t=_{\g} w(r-2)$.
All three of these options result in a contradiction
of the fact that $u$ is a geodesic of length $r$, so
subcase 7.1 can't occur.

\bcase{\it 7.2: If $\gamma \in \{t^2,ta^{\pm 1}\} $:} 
In this subcase, $t^{-1}a^{2s}t$ must be a subword of
$t^{-1}a^it$ again, so
$i=0$ and $w=w_0w_1t^{-1}$.  Note that $\gamma(1)=t$ and
$w \gamma(1) =_{\g} w(r-1)$.  
Then $\gamma$ is a path
of length 2 inside $B(r)$ from $w$ to $u$.  In this
subcase, we may define the path $\delta:=\gamma$.

\bcase{\it 7.3: If $\gamma \in \{a^{\pm 1}\}$:} 
Write $\gamma=a^{k}$ with $|k|=1$.
Recall that $0 \le |i| \le 1$.

If $i=0$, then $t^{-1}a^{2s}t$ must be a subword
of $t^{-1}a^{k}a^{-j}t$, so $2|(k-j)$ and $|j|=1$.
Then $\gamma=a^{\pm j}$.  If $\gamma=a^j$, then
$w=_{\g} u \gamma^{-1} = u(r-2)t^{-1}a^ja^{-j} 
=_{\g} u(r-2)t^{-1}$, contradicting the length $r$
of the geodesic $w$.  Thus $\gamma=a^{-j}$.
The word $\delta:=ta^{-j}t^{-1}a^j=_{\g} a^{-j}$
labels a path from $\ow$ to $\ou$ of length 4.
Since $ w \delta(1) =_{\g} w(r-1)$ and 
$w \delta(2) =_{\g} u(r-2)$, the path $\delta$ stays
inside $B(r)$, and hence satisfies the required properties.
(See Figure \ref{fig:case7}.)
\begin{figure}
\bt{ccc}
\includegraphics[height=4.5cm]{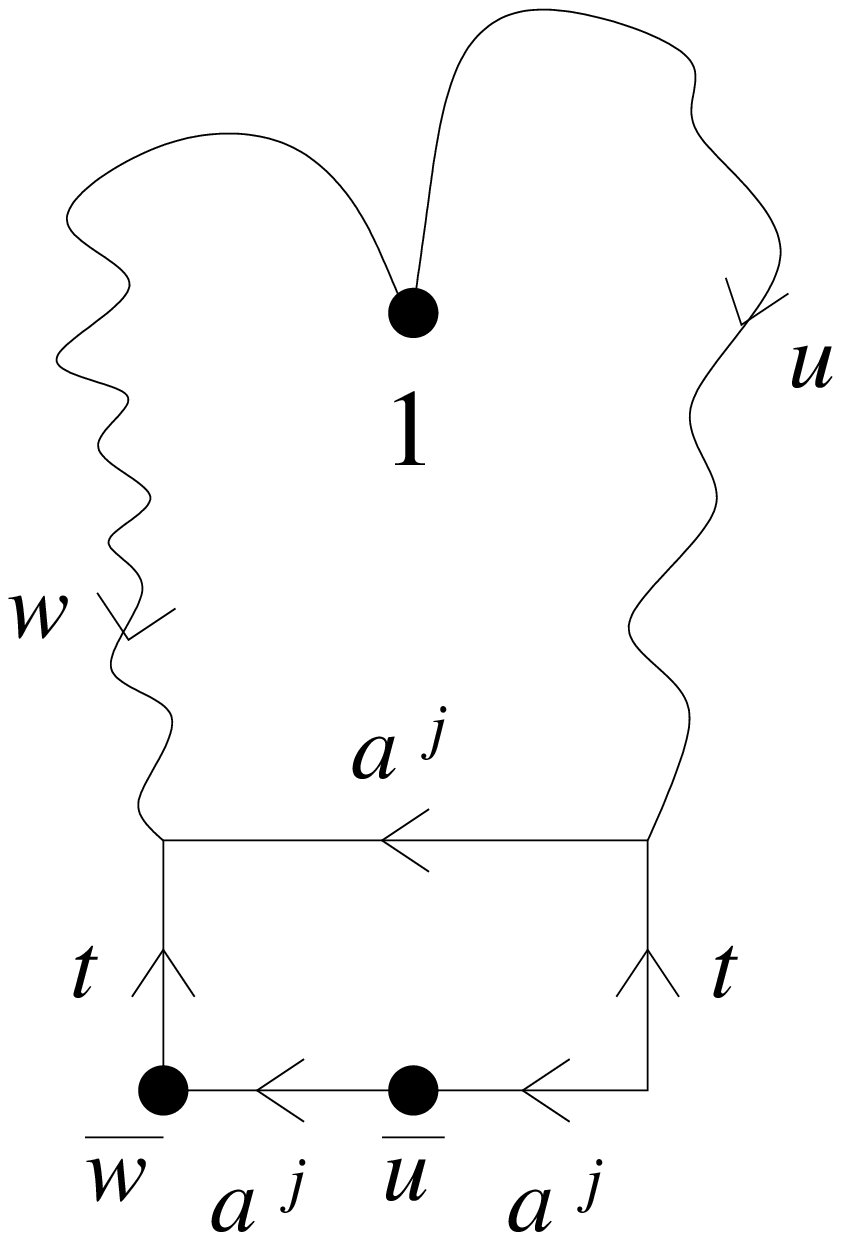} 
& \hspace{5mm} &
 \includegraphics[height=4.5cm]{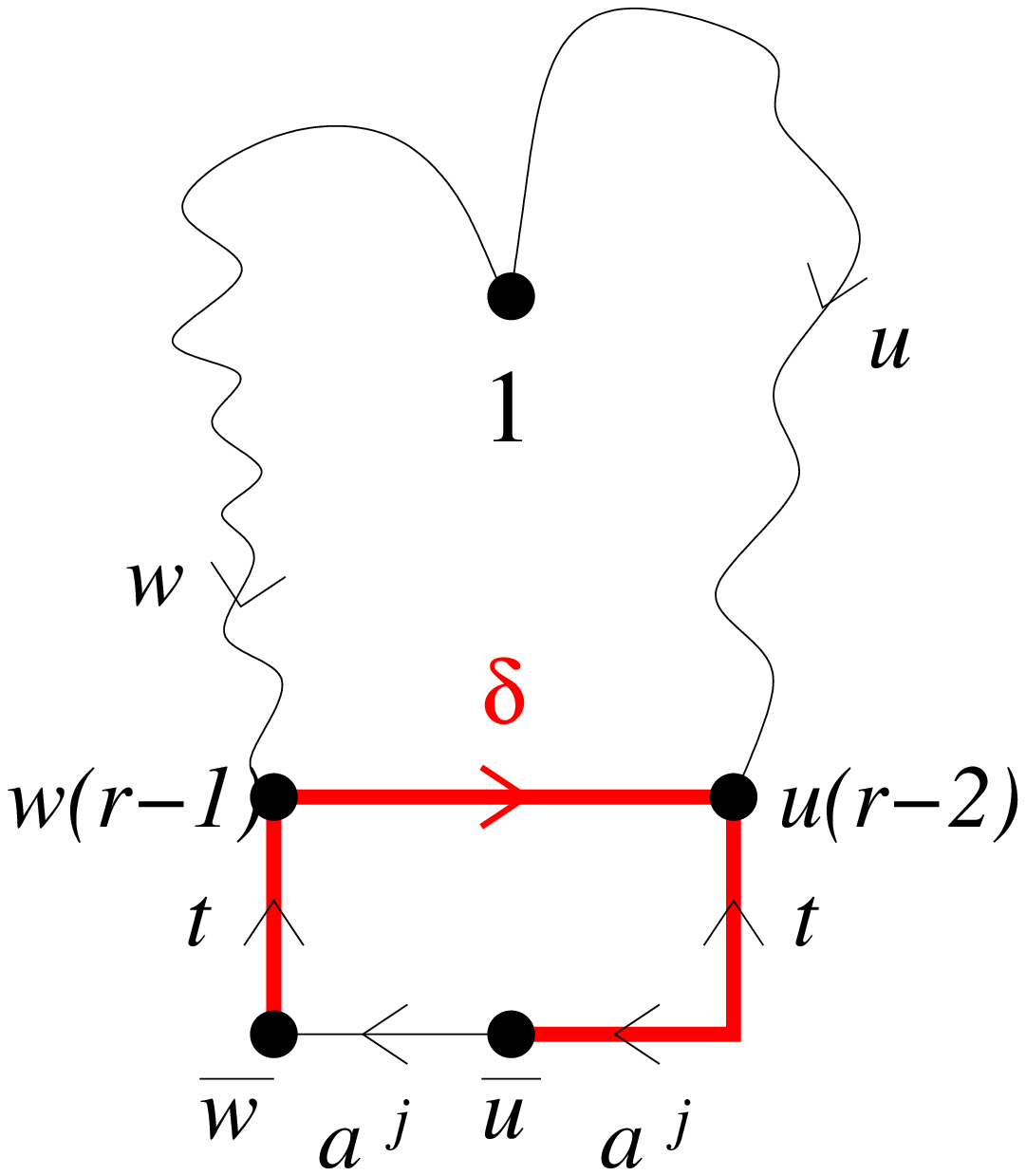} \et
\caption{Case 7.3:  $w=w_0w_1t^{-1}, u=u_0u_1t^{-1}a^j$ .}
\label{fig:case7}
\end{figure}

If $|i|=1$, then we can write $\gamma=a^{\pm i}$.
If $\gamma=a^{- i}$, then $u =_{\g} w \gamma =_{\g} w(r-1)$,
again giving a contradiction; hence $\gamma=a^{i}$.
Note that the word $t^{-1}a^{2s}t$ must be a subword
of $t^{-1}a^{2i}a^{-j}t$, so $2|(2i-j)$ and $j=0$.
Defining $\delta:=a^{-i}ta^{i}t^{-1}=_{\g} a^i$, then $\delta$
labels a path of length 4 from $\ow$ to $\ou$,
with $w\delta(2)=_{\g}w(r-2)$ and $w\delta(3) =_{\g} u(r-1)$,
so the path remains in $B(r)$ as required.

\bcase{\it 7.4: If $\gamma \in \{a^{\pm 2}\}$:}  
Write $\gamma=a^{2k}$ with $|k|=1$.  As in the previous
subcase, we consider the options $i=0$ and $|i|=1$ in 
separate paragraphs.

If $i=0$, then $t^{-1}a^{2s}t$ must be a subword
of $t^{-1}a^{2k}a^{-j}t$, so $j=0$.  Then the length 3 path
labeled by $\delta:=ta^kt^{-1}$ from
$\ow$ to $\ou$ traverses the
vertices represented by $w \delta(1)=_G  w(r-1)$ and
$w \delta(2)=_G u(r-1)$, hence remaining in $B(r)$.

If $|i|=1$, then $\gamma=a^{\pm 2i}$.
If $\gamma=a^{-2i}$, then
$w \gamma(1) =_{\g} w(r-1)$, so
we may define $\delta:=\gamma$.

If $|i|=1$ and $\gamma=a^{2i}$, then 
$t^{-1}a^{2s}t$ must be a subword
of $t^{-1}a^{3i}a^{-j}t$.  Thus $|j|=1$, so $j = \pm i$.
If $j=i$, then $w \gamma(1) =_{\g} u(r-1)$, so 
again
the path $\delta:=\gamma$ has the required properties.
If $j=-i$, then the path of length 6 labeled by
$\delta:=a^{-i}ta^{2i}t^{-1}a^{-i} =_{\g} a^{2i}$
starting at $\ow$ ends at $\ou$.
Since $w \delta(2) =_{\g} w(r-2)$ and 
$w \delta(4) =_{\g} u(r-2)$, this path also remains
within $B(r)$.

\bcase{\it 8: If $w$ is in class $(1)$ and $u$ is in class $(3)$:}  
Then $w \in X$
and $u \in P(X)$. In this case, $\ts(w)=0$, $\ts(u)>0$, and
$\ts(u)=\ts(w)+\ts(\gamma)=\ts(\gamma)$,
so $0<\ts(u)=\ts(\gamma) \le 2$.
Thus $\gamma \in P$, so
$\gamma \in \{t, ta^{\pm 1}, t^2, a^{\pm 1}t \}$.

Suppose $\gamma \in \{t, ta^{\pm 1}, t^2\}$. 
By Proposition \ref{propn:normalform} and Lemma \ref{lem:rlarge}, 
the normal form $w$ can be chosen in the form
$w=w_1t^{-h}$ with
$w_1 \in P$ and $h>10$. 
Then the length $r$ geodesic $u$ cannot represent $wt=_Gw(r-1)$ or
$wt^2=_g w(r-2)$, so 
$\gamma \ne t$ and $\gamma \ne t^2$. For $\gamma=ta^{\pm 1}$,
since $w\gamma(1)=_G w(r-1)$, 
we may define $\delta:=\gamma$.

Suppose for the rest of Case 8 that $\gamma =a^{\pm1}t$ and
write $\gamma=a^mt$ with $|m|=1$.
Proposition \ref{propn:normalform} says that
the normal form $w$ can also be chosen in the form
$w=tw_0t^{-1}a^i$ with $w_0 \in X$ and $0 \le |i| \le 1$.
If $i=m$,  then 
$ u =_{\g} w\gamma=tw_0t^{-1}a^m a^mt =_G w(r-2) a^m$, 
and if $i=-m$, then 
$u =_{\g} tw_0t^{-1}a^{-m} a^{m}t =w(r-2)$, both
contradicting the geodesic length $r$ of $u$.
Then $i=0$ and $w=tw_0t^{-1}$ with $w_0$ in $X$.
We also have $\ts(u)=\ts(\gamma)=1$, and
Lemma \ref{lem:rlarge}
implies that $u \in PX$,
so the normal form $u=a^jtu_0$ with
$u_0$ in $X$ and $|j|\leq 1$.   

If $j=0$ then $w$ and $u$ have a common $t$ prefix, 
and the path $\dd:=tw_0^{-1}u_0$ has the required properties.

Suppose for the remainder of Case 8 that $|j|=1$. 
Then either $\gamma=a^{j}t$ or
$\gamma=a^{-j}t$; 
we consider these two subcases separately.

\bcase{\it 8.1: If $\gamma=a^jt$:} Applying Lemma
\ref{lem:commute} to commute the subwords in parentheses with zero
$t$-exponent-sum,
$$1=_{\g}w\gamma
u^{-1}=tw_0t^{-1}(a^j)(tu_0^{-1}t^{-1})a^{-j}
=_{\g}tw_0u_0^{-1}t^{-1},$$
which yields $w_0=_{\g}u_0$.  
By Proposition \ref{propn:normalform} we can replace
each subword with 
a normal form $w_0=u_0=vta^kt^{-p}$ such that $2\leq |k|\leq 3$
and $v \in P$ with $\ts(v)=p-1$.
Since $r>200$, Lemma \ref{lem:rlarge} implies
that $p>10$.  
Let $s:=\mathrm{sign}(k)$.
Then $w=tvta^{|k|s}t^{-(p+1)}$ and $u=a^jtvta^{|k|s}t^{-p}$.

Consider the path $\delta:=t^pa^{-2s}t^{-p}a^jt^pa^{2s}t^{1-p}$
starting at $\ow$.  
Using Lemma \ref{lem:commute},
$$
\delta=(t^pa^{-2s}t^{-p})(a^j)t^pa^{2s}t^{-(p-1)}=_{\g}
(a^j)(t^pa^{-2s}t^{-p})t^pa^{2s}t^{-(p-1)}=_{\g}a^jt=\gamma,
$$
so $\delta$ labels a path from $\ow$ to $\ou$.
This path $\delta$ both
follows and ``fellow travels'' suffixes of $w$ and $u$;
see Figure \ref{fig:case8-3} for a view of this path, shown
in shading,
when $k$ and $j$ have the same sign.
\begin{figure}
\bt{ccc}
\includegraphics[width=5cm]{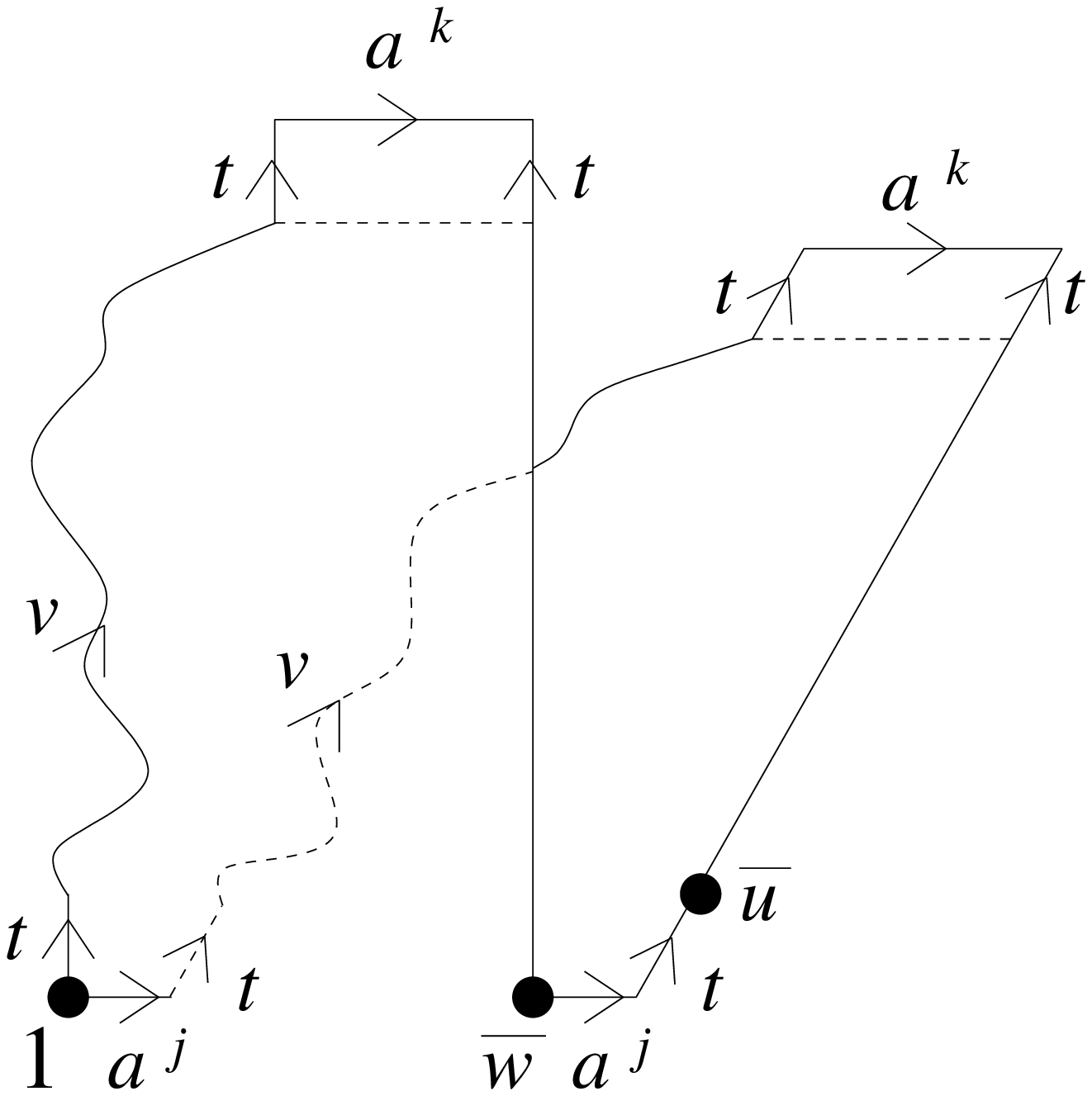} 
& \hspace{5mm} &
\includegraphics[width=5cm]{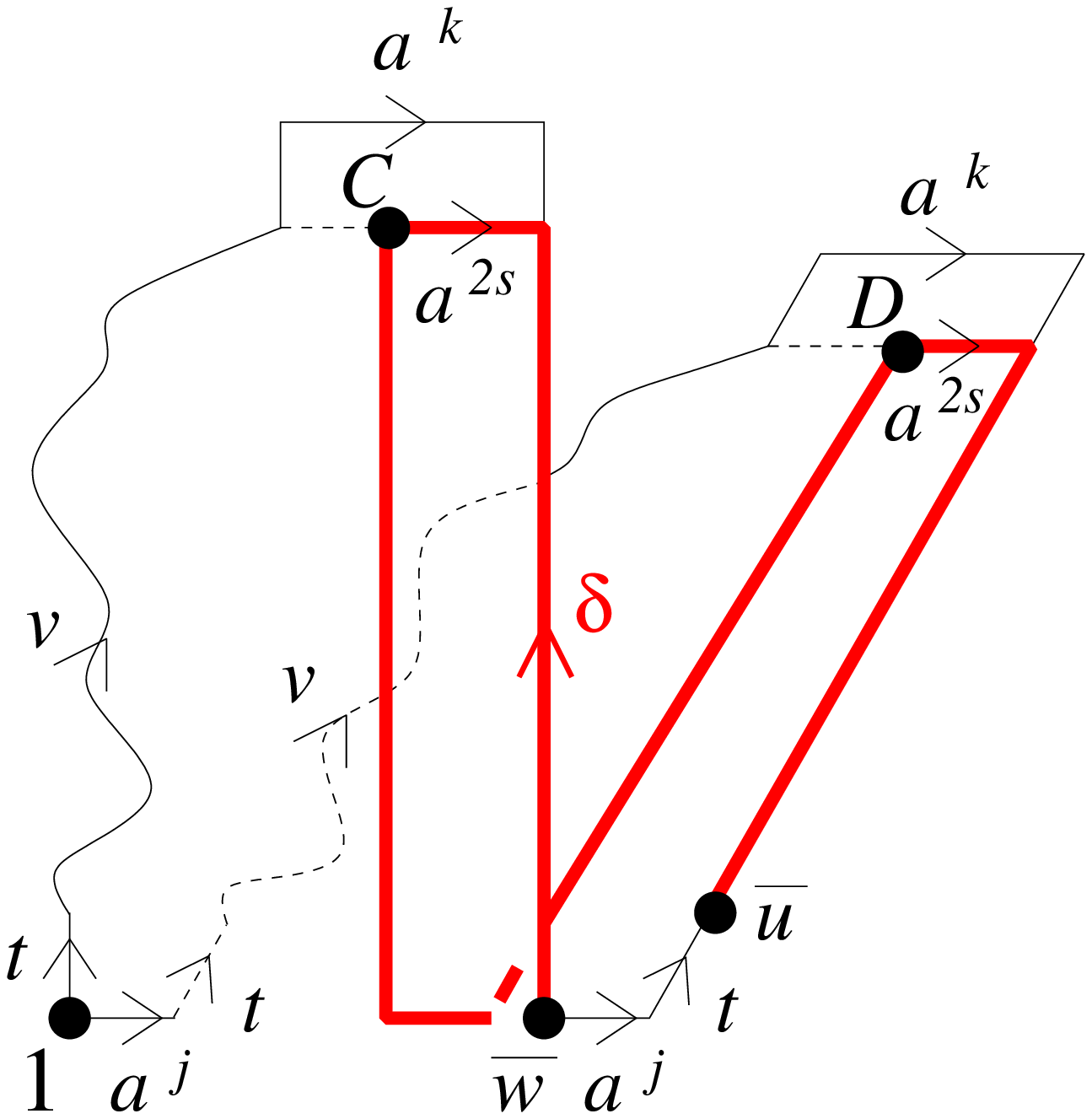} \et
\caption{Case 8.1: $w=tvta^kt^{-p-1},u=a^jtvta^kt^{-p},\gamma=a^jt$.}
\label{fig:case8-3}
\end{figure} 

In order to check that $\dd$ remains in the ball $B(r)$,
we analyze the distances from 1 of several vertices
along the path $\delta$, and together with
the lengths of the subpaths between the vertices.
The prefix $t^p$ of
$\delta$ is the inverse of a suffix of $w$, 
so starting from
$\ow$ the path $\delta$ follows the path $w$
backward.  Then $d(1,\ov{w\delta(i)})=r-i$
for $0 \le i \le p$ and 
$w {\delta(p)}=_{\g} {w(r-p)}$.
The point 
$\ov{w\delta(p+1)}$ must then also lie
in the ball $B(r-(p-1))$.
Since $w {\delta(p+2)} =
w {t^pa^{-2s}} =_{\g} {w(r-p)a^{-2s}}=_{\g}
{w(r-p)ta^{-s}t^{-1}} =_{\g} {w(r-(p+2))t^{-1}}$,
the point $C:=\ov{w\delta(p+2)}$ must lie in
the ball $B(r-(p+1))$.  
Then the initial segment of $\delta$ of length
$p+2$ from $\ow$ to $C$ lies inside $B(r)$.

Similarly, the suffix $t^{-(p-1)}$ of $\delta$
is also a suffix of $u$, so 
$d(1, \ow \ov{\delta(3p+5+i)})=r-(p-1)+i$
for $0 \le i \le p-1$ and
$w {\delta(3p+5)}=_{\g} {u(r-(p-1))}$.
The point 
$\ov{w\delta(3p+4)} \in B(r-(p-2))$.
Since $w {\delta(3p+3)} =_{\g} w{\delta(3p+5)}a^{-2s}
=_{\g} u(r-(p-1))ta^{-s}t^{-1}
=_{\g} a^jtvta^{|k|s}t^{-1}ta^{-s}t^{-1}=_{\g} u(r-(p+1))t^{-1}$,
the point $D:= \ov{w\delta(3p+3)}$ 
must lie in the ball $B(r-p)$.  So the final segment of
$\delta$ of length $p+1$ from $D$ to $\ou$ also
lies in $B(r)$.

Finally, the central section labeled $t^{-p}a^jt^p$
of the path $\delta$
from $C \in B(r-(p+1))$ to $D \in B(r-p)$ has length $2p+1$, 
and hence never leaves the ball $B(r)$.
The entire path $\delta$
has length $4p+4$, whereas $r=l(v)+|k|+p+3\geq
(p-1)+2+p+3=2p+4$, so $l(\delta)\leq 2r-4$.  Thus the path
$\delta$ has the required properties in this subcase.

\bcase{\it 8.2: If $\gamma=a^{-j}t$:} Applying Lemma \ref{lem:commute}
again yields $1=_{\g} w\gamma
u^{-1}=tw_0t^{-1}(a^{-j})(tu_0^{-1}t^{-1})a^{-j}
=_{\g}tw_0t^{-1}tu_0^{-1}t^{-1}a^{-j}a^{-j}
=_{\g}tw_0u_0^{-1}a^{-j}t^{-1}$, implying that $u_0=_{\g}a^{-j}w_0$.  
Plugging this into the expression for $u$ gives
$u=a^jtu_0 =_{\g} a^jta^{-j}w_0 =_{\g} a^{-j}tw_0$.
Note that $r=l(w)=l(w_0)+2$, so $a^{-j}tw_0$ is another
geodesic from $1$ to $\ou$.  Replacing $u$ with
$a^{-j}tw_0$, we can now find the path $\delta$ using
subcase 8.1.

\bcase{\it 9: If $w$ is in class $(2)$ and $u$ is in class $(4)$:}  
Then $w \in (X)N$ and $u$
is either in $NP$,
$XNP$, or $NPX$.
Since $w \not\in (X)NP \cup NPX$, Lemma \ref{lem:twonp} says that
$\ts(w)<\ts(u)$.  
Therefore $\ts(\gamma)>0$, so 
$\gamma \in \{t,t^2,a^{\pm 1}t,ta^{\pm 1}\}$.  
We divide
this case into two subcases, depending on the $t$-exponent sum of $u$.

\bcase{\it 9.1: If $\ts(u)\leq 0$:} By Proposition
\ref{propn:normalform}, the geodesic normal form
$u=u_0u_1t^{-1}a^{m_e}t^f$ with
$u_0 \in X \cup E$, $u_1 \in N$, $|m_e| = 1$, 
and $1 \le f \le e=|\ts(u_1)|+1$.  

\bcase{\it 9.1.1: If $\gamma \in \{t, t^2, a^{\pm 1}t \}$:} 
Then $\gamma$ and $u$ share a suffix $t$, so
$u(r-1)=u\gamma^{-1}(1)$.  
The geodesic $w \ne u(r-1)$, so $\gamma \ne t$.
For $\gamma \in \{t^2,a^{\pm 1}t\}$, 
the path
$\delta:=\gamma$ has the required properties.

\bcase{\it 9.1.2: If $\gamma \in \{ ta^{\pm 1}\}$:} 
Write $\gamma=ta^k$ with $|k|=1$.  Also write
$w=w_0w_1$ with $w_0 \in X \cup E$ and $w_1 \in N$.
Then  

\hspace{0.3in}
$1=_G\tilde u \gamma^{-1} \tilde w^{-1}=
\tilde u_0 u_1t^{-1}a^{m_e}t^fa^{-k}t^{-1}w_1^{-1}\tilde w_0^{-1}$

\hspace{0.5in} $=_{\g}
\tilde u_0 u_1t^{-1}a^{m_e}t^{f-1}a^{-2k}w_1^{-1}\tilde w_0^{-1}\in NP$.

\noindent
Lemma \ref{lem:britton} implies that the latter
word must contain a non-geodesic $t^{-1}a^{2s}t$
subword. Then the first occurrence of a $t$ must be in $w_1^{-1}$, so
$f=1$.  Write
$u=u_0u_1t^{-1}a^{m_e}t=u_0u_1't$ with $u_1':=u_1t^{-1}a^{m_e} \in N$.
Note that $w =_G u\gamma^{-1} =_G u_0u_1'ta^{-k}t^{-1}
=_G u(r-1)a^{-2k}$.   
Let $v:=u_0u_1'a^{-k}=u(r-1)a^{-k}=_G wa^k$.  
The  vertex $\ov{v} \in B(r)$. 
If $\ov{v} \in B(r-1)$,
then the path $\delta:=a^{2k}t$ from $\ow$
to $\ou$ satisfies $w\delta(1) =_{\g} v$ and
$w\delta(2) =_{\g} u(r-1)$, so $\delta$ is a path of
length 3 inside $B(r)$ from $\ow$
to $\ou$.
On the other hand, if $\ov{v} \not\in B(r-1)$,
then $v$ is a length $r$ geodesic 
in $(X)N$, and $w=_{\g} va^{-k}$ so $d(\ov{v},\ou)=1$.  Applying case 7.3
to the geodesics $v$ and $w$ in class $(2)$,
there is a path $\delta'$ of length 4 inside $B(r)$ from
$\ow$ to $\ov{v}$.
Let $\delta:=\delta'a^kt$.  Then $\delta$
is a path of length 6 from $\ow$ to $\ou$ 
inside $B(r)$.

\bcase{\it 9.2: If $\ts(u)> 0$:} Since $\ts(w)<0$ and
$\ts(w)+\ts(\gamma)=\ts(u)$ then we must have 
$\ts(w)=-1$, $\ts(u)=1$, and $\gamma = t^{2}$.  
By Proposition \ref{propn:normalform} and Lemma \ref{lem:rlarge}, 
$w=w_0t^{-1}a^i$
with $w_0 \in X$, $|i| \le 1$,  
$w_0=vta^kt^{-p-1}$ 
with $v \in P$, $\ts(v)=p>9$, and $2 \le |k| \le 3$.
Since $u =_G w\gamma=vta^kt^{-p}t^{-2}a^it^2$ and $u$ has geodesic length 
$r=l(w)=l(v)+|k|+|i|+p+3$, then $i \ne 0$.
Using Lemma \ref{lem:commute},
$u =_G (t^{-1}a^it)(vta^kt^{-p-1})t$.
The word $x:=t^{-1}a^itvta^kt^{-p}$ is 
another geodesic labeling a path from the identity to 
$\ou$.

Let $s:=$sign$(k)$, and define the path
$$
\delta:=
a^{-i}t^{p+1}a^{-2s}t^{-(p-1)}t^{-2}a^it^2t^{p-1}a^{2s}t^{-(p-1)}
$$
starting at $\ow$.
Using Lemma \ref{lem:commute},

$\delta=_{\g}
a^{-i}t^{p+1}a^{-2s}t^{-(p-1)}(t^{p-1}a^{2s}t^{-(p-1))(t^{-2}a^it^2)}
=_{\g}t^{2}=\gamma,$

\noindent 
so $\delta$ labels a path 
from $\ow$ to $\ou=\ov{x}$.
The length $l(\delta)=4p+8$, and
the length $r=l(w)=l(v)+|k|+p+4 \ge 2p+6$,
so $l(\delta) \le 2r-4$.
  
The proof that $\dd$ remains in $B(r)$ is
similar to case 8.1.
In particular, note that
$\ov{w\dd(p+2)}=\ov{w(r-(p+2))} \in B(r-(p+2))$.
Since $w\delta(p+4)=
(vta^kt^{-p}t^{-2}a^i)(a^{-i}t^{p+1}a^{-2s})=_G
vta^{k-s}t^{-1}=_{\g}
w(r-(p+4))t^{-1}$,
the point $\ov{w\delta(p+4)} \in B(r-(p+3))$.
Also 
$\ov{w\delta(3p+9)} = \ov{x(r-(p-1))} \in B(r-(p-1))$.  Finally,
since $w\delta(3p+7) =_{\g} x(r-(p-1))a^{-2s}=_{\g}
t^{-1}a^itvta^kt^{-1}a^{-2s}=_G t^{-1}a^itvta^{k-s}t^{-1}
=_G x(r-(p+1))t^{-1}$,
then the point $\ov{w\delta(3p+7)} \in B(r-p)$.
Then the five successive
intermediate subpaths of $\dd$ between $\ow$, these four points, 
and $\ou$
are too short to allow $\dd$ to leave $B(r)$.

\bcase{\it 10: If both $w$ and $u$ are in class $(4)$:}  
In this case both $w$ and $u$ are in $(X)NP \cup NPX$.
We may assume without loss of
generality that $\ts(w) \leq \ts(u)$. It follows that
$\ts(\gamma) \geq 0$, so $\gamma \in \{a^{\pm 1},a^{\pm 2}, a^{\pm
1}t, ta^{\pm 1}, t, t^2\}$.

We divide this case into three subcases, depending on the $t$-exponents
of $w$ and $u$.

\bcase{\it 10.1: If $\ts(w) \geq 0$ and $\ts(u) \geq 0$:}
In this case Proposition \ref{propn:normalform} 
says that we have geodesic normal forms
$u,w \in NP(X)$, and moreover $w=t^{-p_1}w'$ and
$u=t^{-p_2}u'$ with $p_1>0$, $p_2>0$, and $w', u' \in P(X)$. 
Thus $w(1)=t^{-1}=u(1)$, and we may define
$\dd:=w'^{-1}t^{p_1-1}t^{-(p_2-1)}u'$.

\bcase{\it 10.2: If $\ts(w) <0$ and $\ts(u) \leq 0$:}  
In this subcase, we have normal forms $w,u \in (X)NP$,
and we can write 
$$
w=w_0w_1t^{-1}a^{i_1}t^{f_1} 
\hspace{0.4in} {\rm and} \hspace{0.4in}
u=u_0u_1t^{-1}a^{i_2}t^{f_2}
$$
with $w_0,u_0 \in X \cup E$,
$w_1 \in N$,
$u_1 \in N \cup E$,
$f_1 \ge 1$, $f_2 \ge 1$,
$\ts(w_1)\leq -f_1$,
$\ts(u_1)\leq -(f_2-1)$, and 
$|i_1|=|i_2|=1$.
Lemma \ref{lem:twonp} implies that $\ts(w_1)=\ts(u_1)$.
Then $\ts(w\gamma)=\ts(w_1)-1+f_1+\ts(\gamma)=\ts(u)=\ts(u_1)-1+f_2$,
so $f_2=f_1+\ts(\gamma)\ge f_1$.

\bcase{\it 10.2.1: If $\gamma \in \{t, t^2, a^{\pm 1}t\}$:} 
Since the last letter of $u$ is $t$, the proof of
Case 9.1.1 shows that $\gamma \ne t$, and for
$\gamma \in \{t^2, a^{\pm 1}t\}$, we may define 
 $\delta:= \gamma$.

\bcase{\it 10.2.2: If $\gamma \in \{a^{\pm 1}\}$:} 
Write $\gamma=a^{k}$ with $|k|$=1.
The word $\delta:=t^{-1}a^{2k}t$ labels
a path of length 4 from $\ow$ to $\ou$.
Since both words $w$ and $u$ end with a $t$, then
$w \delta(1)=_{\g} w(r_1)$ and $w \delta(3)=_{\g} u(r-1)$, hence
$\delta$ lies in $B(r)$.

\bcase{\it 10.2.3: If $\gamma \in \{ta^{\pm 1}\}$:}
Write $\gamma=ta^k$ with $|k|=1$.
In this subcase, $f_2=f_1+\ts(\gamma)=f_1+1 \ge 2$.
Then the word $w$
ends with a $t$ and $u$ ends with $t^2$.
Now $u=u(r-1)t =_G w\gamma=wta^k=_G wa^{2k}t$.
Let $v:=u(r-1)a^{-k}=_G wa^k$.  Then $v \in (X)NP$ and $\ov{v} \in B(r)$.
The remainder of the proof in this subcase is similar to Case 9.1.2.
If $v \in B(r-1)$, then $\delta:=a^{2k}t$ has the required properties.
If $v \not\in B(r-1)$, then Case 10.2.2 provides
a path $\delta'=t^{-1}a^{2k}t$ inside $B(r)$ from $\ow$ to $\ov{v}$, and
the path $\delta:=\delta'a^kt=t^{-1}a^{2k}ta^kt$ 
from $\ow$ to $\ou$ satisfies the
required conditions.

\bcase{\it 10.2.4: If $\gamma \in \{a^{\pm 2}\}$:}
Write $\gamma=a^{2k}$ with $|k|=1$.
In this case we have $\ts(u)=\ts(w)<0$ and
$f_2=f_1+\ts(\gamma)=f_1$.

The radius $r=l(w)=l(w_0)+l(w_1)+2+f_1 \ge \ts(w_1)+2+f_1 \ge 2f_1+2$.
If $r=2f_1+2$, then $w=t^{-f_1-1}a^{i_1}t^{f_1}$ and 
$u=t^{-f_1-1}a^{i_2}t^{f_1}$.  Since $\ow \ne \ou$,
then $i_1 \ne i_2$, so $i_2=-i_1$.  Since $r>200$, then $f_1>1$.
Now
$u^{-1}w\gamma=_G
(t^{-f_1-1}a^{i_2}t^{f_1})^{-1}(t^{-f_1-1}a^{i_1}t^{f_1})a^{2k}=_{\g}
t^{-(f_1-1)}a^{i_1}t^{f_1-1})a^{2k}\ ;$
according to Britton's Lemma, this last expression cannot equal
the trivial element 1 in $G$.  
Thus the radius $r \ne 2f_1+2$,
so $r \ge 2f_1+3$.

Define $\delta:=t^{-f_1}(a^{-i_1})(t^{f_1}a^{2k}t^{-f_1})a^{i_1}t^{f_1}$.
Using Lemma \ref{lem:commute} to commute the subwords in 
parentheses, and freely reducing the resulting word, shows that
$\delta =_{\g}  a^{2k}=\gamma$, so $\delta$ labels a path from $\ow$
to $\ou$.  
The length $l(\gamma)=4f_1+4=2(2f_1+3)-2\le 2r-2$.
The prefix $\delta(f_1+2)=w^{-1}(f_1+2)$ is the inverse of
a suffix of $w$, so 
$\ov{w\dd(f_1+2)}=\ov{w(r-(f_1+2))} \in B(r-(f_1+2))$. 
The word $t^{f_i}$ is a suffix of both $\dd$ and $u$,
so $\ov{w\dd(3f_1+4)}=\ov{u(r-f_1)} \in B(r-f_1)$.
The three subpaths of $\dd$ between $\ow$, the two points above,
and $\ou$ are again too short to allow $\dd$ to leave $B(r)$.

\bcase{\it 10.3: If $\ts(w) <0$ and $\ts(u) > 0$:}  
Then $\gamma=t^2$, $\ts(w) =-1$, and $\ts(u) =1$.  
From Proposition \ref{propn:normalform}, the normal
form $$w=w_0t^{-1}a^{m_1}t^{-1} \cdots t^{-1}a^{m_p}t^{p-1}$$
with either $w_0 \in X$ or $w_0=a^k \in E$ for some $|k| \le 3$,
$p \ge 2$, and 
$|m_p|=1$.  Using Lemma \ref{lem:commute}, then the word
$\check w := w_0t^{-p}a^{m_p}ta^{m_{p-1}} \cdots ta^{m_1}$ is
another geodesic representative of $\ow$.
The normal form for $u \in NP(X)$ has the form
$u=t^{-e}a^jtu_1u_0$ with $e \ge 1$, $|j|=1$, $u_1 \in P$
with $\ts(u_1)=p$, and $u_0 \in X \cup E$.
Lemma \ref{lem:twonp} shows that
$p=e$.  Replacing $w$ by the alternate normal form $\check w$,
then we can write 
$$
w=w_0t^{-p}a^{i}tw_2
\hspace{0.4in} {\rm and} \hspace{0.4in}
u=t^{-p}a^jtu_1u_0
$$
such that
either $w_0 \in X$ or $w_0=a^k$ for $|k| \le 3$,
$p \ge 2$, 
$|i|=1$,
$w_2 \in P \cup E$ with $\ts(w_2)=p-2$,
$|j|=1$, $u_1 \in P$
with $\ts(u_1)=p$, and $u_0 \in X \cup E$.

We will divide case 10.3 into further subcases, depending
on the form of $w_0$ and the length of $w_2$.

\bcase{\it 10.3.1: If $w_0 \in X$:}
In this case Proposition \ref{propn:normalform}
says that we can write $w_0=w_3ta^kt^{-l}$ with
$l \ge 1$, $w_3 \in P \cup E$, $\ts(w_3)=l-1$, and $2 \le |m| \le 3$.
Let $s = \pm 1$ be the sign of $m$, so that
$m=|m|s$.
Then 
$$
w=w_3ta^{|m|s}t^{-l}t^{-p}a^itw_2\ .
$$
The radius 
$r=l(w_3)+l(w_2)+|m|+l+p+3 \ge \ts(w_3)+l(w_2)+l+p+5=
l(w_2)+2l+p+4$.

Applying Lemma \ref{lem:commute}, we obtain
$u=_{\g} \check w t^2=  (w_3ta^{|m|s}t^{-l})(t^{-p}a^itw_2t)t
=_{\g} t^{-p}a^itw_2tw_3ta^{|m|s}t^{-(l-1)}.$  
Then 
$$
\check u:= t^{-p}a^itw_2tw_3ta^{|m|s}t^{-(l-1)}\ .
$$  
is another geodesic representative of $\ou$.

\bcase{\it 10.3.1.1: If $l \ge 2$:}  
Define 
$$
\delta:=(w_2^{-1}t^{-1}a^{-i}t^{p-1})(t^{l}a^{-2m}t^{-l})
 t^{-(p-1)}a^itw_2t^{l}a^{2m}t^{-(l-2)}\ .
$$
Applying Lemma \ref{lem:commute} to the subwords in parentheses
shows that $\delta =_G t^2=\gamma$, so $\delta$ labels a path 
from $\ow$ to
$\ov{\check u}=\ou$.
The length $l(\delta)=2l(w_2)+4l+2p+4 \le 2r-4$.

The vertex 
$\ov{w\delta(l(w_2)+p+l+1)}=\ov{w(r-(l(w_2)+p+l+1))} 
\in B(r-(l(w_2)+p+l+1))$.
Now $w\delta(l(w_2)+p+l+3)=_{\g}
w_3ta^{|m|s}t^{-1}a^{-2s}=_{\g} w_3ta^{(|m|-1)s}t^{-1}
w(r-(l(w_2)+p+l+3))t^{-1}$, implying
 $\ov{w\delta(l(w_2)+p+l+3))} \in B(r-(l(w_2)+p+l+2))$.
The suffix $t^{-(l-2)}$ of $\delta$ is also a
suffix of $\check u$, so 
$\ov{w\delta(2l(w_2)+3l+2p+6)} = \ov{\check u(r-(l-2))} \in B(r-(l-2))$.
Finally,
$w\delta(2l(w_2)+3l+2p+4)=_G
\check u(r-(l-2))a^{-2s}=_{\g} 
t^{-p}a^itw_2tw_3ta^{|m|s}t^{-1}a^{-2s}=_{\g} 
t^{-p}a^itw_2tw_3ta^{(|m|-1)s}t^{-1}$, so
$\ov{w\delta(2l(w_2)+3l+2p+4)}=\ov{\check u(r-l)t^{-1}} \in B(r-(l-1))$.
The five subpaths of $\dd$ between $\ow$, these
four points, and $\ou$ are each too short to leave $B(r)$.

\bcase{\it 10.3.1.2: If $l=1$:}  
In this case define 
$$
\delta:=(w_2^{-1}t^{-1}a^{-i}t^{p-1})(ta^{-2s}
t^{-1})t^{-(p-1)}a^itw_2t^2a^{s}\ .
$$
Commuting the subwords in parentheses, then $\dd=_Gt^2=\gamma$ and
$\delta$ labels a path from $\ow$ to $\ou$.
The length $l(\delta)=2l(w_2)+2p+9=2l(w_2)+4l+2p+5 \le 2r-3$.

The proof that $\dd$ remains in $B(r)$ is similar to
Case 10.3.1.1.  In particular,
$\ov{w\delta(l(w_2)+p+2)} =\ov{w(r-(l(w_2)+p+2))} \in B(r-(r-(l(w_2)+p+2))$,
$\ov{w\delta(l(w_2)+p+4)} =\ov{w(r-(l(w_2)+p+4))t^{-1}} 
 \in B(r-(r-(l(w_2)+p+3))$, 
$\ov{w\delta(2l(w_2)+2p+8)} = \ov{\check u(r-1)} \in B(r-1)$, and
the four successive subpaths between
$\ow$, these three points, and $\ou$ are too short for $\dd$
to leave $B(r)$.

\bcase{\it 10.3.2: If $w_0=a^k$ with $|k| \le 3$, and $l(w_2) = p-2$:}
Since $\ts(w_2) =p-2$, then  $w_2=t^{p-2}$ and
$w=a^kt^{-p}a^it^{p-1}$ with $p \ge 2$ and $|i|=1$.
Recall that $u=t^{-p}a^jtu_1u_0$ with $|j|=1$ and $\ts(u_1)=p$.
The radius $r=|k|+2p=p+2+l(u_1u_0)$, so
$|k| =  l(u_1u_0) -p+2 \ge 2$.

If $|k|=2$, then $l(u_1u_0)=p$ and
$u=t^{-p}a^jt^{p+1}$.  The trivial element
$1=_{\g} wt^2u^{-1}=a^kt^{-p}a^it^{p-1}t^2t^{-(p+1)}a^{-j}t^p
=_{\g} a^kt^{-p}a^ia^{-j}t^p$.  Since $a^k \ne_{\g} 1$, $i \ne j$.
But $a^kt^{-p}a^ia^{i}t^p=_{\g} a^kt^{-(p-1)}a^it^{(p-1)}$, and
Britton's Lemma says that the latter word cannot represent
the trivial element 1, so $i \ne -j$.  Therefore
we cannot have $|k|=2$.

If $|k|=3$, then $l(u_1u_0)=p+1$ and $\ts(u_1)=p$,
so the word $u_1u_0$ contains one occurrence of $a$ or $a^{-1}$.
Write $u_1u_0=t^ba^lt^{p-b}$ for some $0 \le b \le p$ and $|l|=1$.
Then $wt^2u^{-1}$ freely reduces to 
$a^kt^{-p}a^it(t^{b}a^{-l}t^{-b})(a^{-j})t^p$.  Commuting
the parenthetical subwords and reducing again gives
$1=_G wt^2u^{-1} =_{\g} a^kt^{-p}a^ita^{-j}t^{b}a^{-l}t^{p-b}$.  
This word is in $NP$, and doesn't contain a subword
of the form $t^{-1}a^{2m}t$, so Britton's Lemma
(or Lemma \ref{lem:britton}) implies a contradiction again.
Therefore case 10.3.2 cannot occur.

\bcase{\it 10.3.3: If $w_0=a^k$ with $|k| \le 3$ and $l(w_2) \ge p-1$:}
In this case we have the geodesic normal forms
$w=a^kt^{-p}a^itw_2$ and $u=t^{-p}a^jtu_1u_0$.
The radius $r=l(w_2)+|k|+p+2=l(u_1u_0)+p+2$.
We have two subcases, depending on whether $i=j$ or $i \ne j$.

\bcase{\it 10.3.3.1: If $i=j$:}
In this subcase note that
$\gamma=t^2 =_{\g} w^{-1}u=
w_2^{-1}t^{-1}(a^{-j})(t^pa^{-k}t^{-p})a^jtu_1u_0$.
Applying Lemma \ref{lem:commute} 
and reducing shows that
the word $\delta:=w_2^{-1}t^{p-1}a^{-k}t^{-(p-1)}u_1u_0 =_G \gamma$,
and so $\dd$ 
labels a path from $\ow$ to $\ou$.
The length $l(\delta)=2r-6$.

As usual we analyze the distances from 1 of several vertices
along the path $\delta$.
$\ov{w\delta(l(w_2))}=\ov{w(r-l(w_2))} \in B(r-l(w_2))$,
$\ov{w\delta(l(w_2)+2p-2+|k|)} = \ov{u(r-l(u_1u_0))} \in B(r-l(u_1u_0))$,
and the three intervening subpaths are each too short to
allow $\dd$ to leave $B(r)$.

\bcase{\it 10.3.3.2: If $i=-j$:}
Similar to Case 10.3.3.1, after commuting and reduction we have
$\gamma =_{\g} w^{-1}u
 =_{\g} w_2^{-1}t^{p-1}a^{-k}t^{-p}a^{j}a^{j}tu_1u_0$.
Then the
word $\delta:=w_2^{-1}t^{p-1}w_0^{-1}t^{-(p-1)}a^{j}u_1u_0=_G \gamma$
labels a path from $\ow$ to $\ou$ and
has length $l(\delta)=2r-5$.

As in Case 10.3.3.1, we have
$\ov{w\delta(l(w_2))}=\ov{w(r-l(w_2))} \in B(r-l(w_2))$ and
$\ov{w\delta(l(w_2)+2p-1+|k|)}= \ov{u(r-l(u_1u_0))} \in B(r-l(u_1u_0))$.
Now $w\delta(l(w_2)+2p-2+|k|)=_{\g}
(t^{-p}a^{j}t)a^{-j}=_{\g}t^{-p}a^{-j}t$, so
$\ov{w\delta(l(w_2)+2p-2+|k|)} =\ov{u(r-(l(u_1u_0)+2))a^{-j}t}
\in B(r-l(u_1u_0))$ as well.
The four successive subpaths of $\dd$ between $\ow$,
these three vertices, and $\ou$ have lengths too short
to allow $\dd$ to leave $B(r)$. Therefore
$\delta$ has the required properties.

\vspace{3mm}

Therefore in all of cases 1-10, either 
the case cannot occur or
the path
$\delta$ with the required properties can be
constructed, completing the proof of Theorem \ref{thm:BS12mac}.
\end{proof}


\section{Non-convexity properties for $BS(1,q)$}

In the first half of this section we show, in Theorem
\ref{thm:notpac}, that the group
$G:=BS(1,2)=\langle a,t ~|~ tat^{-1}=a^2\rangle$ 
with generators 
$A:=\{a,a^{-1},t,t^{-1}\}$ does not satisfy Poenaru's
$P(2)$ almost convexity condition.  
We start
by defining some notation.
Let $n$ be an arbitrary natural number with $n>100$
and let $w:=t^{n}a^2t^{-n}$ and
$u:=at^{n}a^2t^{-(n-1)}$ (see Figure  \ref{fig:notPAC}).
 Then $w$ and $u$ are words
of length $R:=2n+2$.  Moreover, using Lemma \ref{lem:commute},
$wat=(a)(t^{n}a^2t^{-n})t=_Gu$, so
$d(\ow,\ou)=2$ and the word $\gamma:=at$ labels a
path from $\ow$ to $\ou$. 

 \begin{figure}
\bt{c}\includegraphics[height=4.2cm]{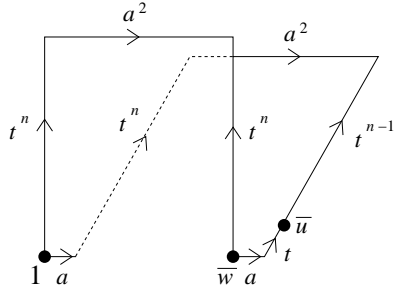}
\et
\caption{$w=t^{n}a^2t^{-n},
u=at^{n}a^2t^{-(n-1)}$}
\label{fig:notPAC}
\end{figure}

\begin{lem}\label{lem:boundingm}
If $m \in \Z$ and $\overline{a^m}$ is in the ball $B(R)=B(2n+2)$
in the Cayley graph of $G$, then either $m=2^{n+1}$ or 
$m \le 2^n+2^{n-1}+2^{n-2}$.
\end{lem}

\begin{proof} 
For $\ov{a^m} \in B(R)$,
Proposition \ref{propn:normalform} says 
that there is a geodesic word
$V$ in the normal form 
$v=t^ha^st^{-1}a^{k_{h-1}}t^{-1} \cdots t^{-1}a^{k_0}$
with $2 \le |s| \le 3$ and each $|k_i| \le 1$, such that
$v =_G a^m$.  This word contains
$2h$ letters of the form $t^{\pm 1}$ and $l(v) \le 2n+2$, so $h \le n$.
Also, $v =_G a^{2^hs+2^{h-1}k_{h-1}+\cdots+k_0} =_G a^m$,
so $m=2^hs+2^{h-1}k_{h-1}+\cdots+k_0$.

If $h=n$, then $v=t^na^{\pm 2}t^{-n}$, so $m = \pm 2^{n+1}$.
If $h=n-1$, then there are at most 4 occurrences
of $a^{\pm 1}$ in the expression for $v$; that is,
$|s| +\sum|k_i| \le 4$.
The value of $m$ will be maximized if $s=+3$, $k_{h-1}=+1$,
and $k_i=0$ for all $i \le h-2$;
in this case, $m=(3)2^{n-1}+2^{n-2}=2^n+2^{n-1}+2^{n-2}$.

Finally, if $h \le n-2$, then
$m \le 2^{n-2}(3)+2^{n-3}(1)+\cdots+(1)={{2^n-1} \over {2-1}}$, so
$m<2^n+2^{n-1}+2^{n-2}$.
\end{proof}

\begin{lem}\label{lem:wugeod}
The words $w=t^{n}a^2t^{-n}$, $wa$, $wa^{-1}$, $wt^{-1}$, and
 $u=at^{n}a^2t^{-(n-1)}$ label geodesics
in the Cayley graph of $G$.
\end{lem}

\begin{proof} 
As a consequence of Lemma \ref{lem:boundingm}, the
vertices $\ov{a^{2^{n+1} + 1}}$ and $\ov{a^{2^{n+1} - 1}}$ 
are not in the ball $B(R)$.
The words $t^na^2t^{-n}a=wa$ and $aw$ both label paths
from the identity to $\ov{a^{2^{n+1} + 1}}$, and  the
word $wa^{-1}$ labels a path from 1 to $\ov{a^{2^{n+1} - 1}}$.
Each of these words has length
$2n+3=R+1$, so all three paths must be geodesic.  As a consequence,
the subwords $w$ of $wa$ and $u$ of $aw$ are also geodesic.

The element $\ov{wt^{-1}}$ has a geodesic normal form from
Proposition \ref{propn:normalform} given by
$v=t^ha^st^{-1}a^{k_{h-1}}t^{-1} \cdots t^{-1}a^{k_0}t^{-1}a^{l}$
with $|l| \le 1$.  Since $a^{2^{n+1}}t^{-1}v^{-1} =_G 1$, 
Lemma \ref{lem:britton} shows that $a^l=a^{2^i}$ with $i \in \Z$,
so $l=0$.  Hence $wt^{-1}$ is also geodesic.
\end{proof}

\begin{thm}\label{thm:notpac}
The group $G=BS(1,2)=\langle a,t ~|~ tat^{-1}=a^2\rangle$ 
is not $P(2)$ with respect
to the generating set $A=\{a,a^{-1},t,t^{-1}\}$.
\end{thm}

\begin{proof}
Let $n \in \N$ with $n>100$,
$w=t^{n}a^2t^{-n}$,
$u=at^{n}a^2t^{-(n-1)}$, and $R=2n+2$.  Then
$\ow$ and $\ou$ lie in the sphere $S(R)$ and $d(\ow,\ou)=2$.
Let $\delta$ be a path inside the ball $B(R)$ from $\ow$ to $\ou$
that has minimal possible length.
In particular, $\delta$ does not have any subpaths 
that traverse a single vertex more than once.

From Lemma \ref{lem:wugeod},
$\ov{wa^{\pm 1}}$ and $\ov{wt^{-1}}$ are not in $B(R)$, so the
first letter of the path $\delta$ must be $t$.  
Let $\pi:\C \ra T$ denote the horizontal projection map
from the Cayley complex of $G$ to the regular tree $T$ of valence 3,
as described at the beginning of Section 2.
The vertices $\pi(\ov{w\delta(1)})=\pi(\ov{t})$ and 
$\pi(\ou)=\pi(\ov{at})$ are
the terminal vertices of the two distinct edges of $T$
with initial vertex $\pi(\ow)=\pi(1)$.  Since the projection
of the path $\delta$ begins at $\pi(1)$, goes to $\pi(\ov{t})$, and
eventually ends at $\pi(\ov{at})$, there must be another
point $P:=\ov{w\delta(j)}$ along the path $\delta$
with $\pi(P)=\pi(\ov{w\delta(j)})=\pi(1)$
and $1 < j < l(\delta)$.
Let $\delta_1$ be the subpath 
of $\delta$ from $\ov{w}$ to $P$.

Our assumption that $\delta$ has minimal possible length
implies that $P \ne \ow$.  Since $\pi(P)=\pi(1)$, $P =_G a^m$
for some $m \in \Z$.
Then Lemma \ref{lem:boundingm} shows that
$m \le 2^n+2^{n-1}+2^{n-2}$.  Since $\delta_1$ labels a path from
$\ov{a^{2^{n+1}}}$ to $\ov{a^m}$, then
$\delta_1^{-1} =_G a^{2^{n+1}-m}=a^k$ with 
$k=2^{n+1}-m \ge 2^{n-2}>2^{(n-4)+1}$.
Applying the contrapositive of Lemma \ref{lem:boundingm}, 
$\ov{a^k}$ is not in the ball $B(2(n-4)+2)$, so
the length $l(\delta_1^{-1}) > 2(n-4)+2$.
Therefore $l(\delta)>R-8$, so this length cannot be bounded above 
by a sublinear function of $R$.
\end{proof}

For the remainder of this section, let 
$G_q :=  BS(1,q)=\langle a,t ~|~ tat^{-1}=a^q\rangle$
with $q \ge 7$ and  
with generators 
$A:=\{a,a^{-1},t,t^{-1}\}$.  We will apply methods very similar
to those developed above, to show that these groups are not \mac.

\begin{lem}\label{lem:boundmq}
If $m \in \Z$ and $\overline{a^m}$ is in the ball $B(R)=B(2n+1)$
in the Cayley graph of $G_q$, then either $m=q^{n}$ or 
$m \le 3q^{n-1}$.
\end{lem}

\begin{proof} 
For $\ov{a^m} \in B(R)$,
Proposition \ref{propn:normalform} says 
that there is a geodesic word
$V$ in the normal form 
$v=t^ha^st^{-1}a^{k_{h-1}}t^{-1} \cdots t^{-1}a^{k_0}$
with $1 \le |s| \le q-1$ and each $|k_i| \le \lfloor {q \over 2} \rfloor$, 
such that
$v =_G a^m$.  Then $h \le n$ and
$m=q^hs+q^{h-1}k_{h-1}+\cdots+k_0$.

If $h=n$, then $v=t^na^{\pm 1}t^{-n}$, and $m = \pm q^{n}$.
If $h=n-1$, then there are at most 3 occurrences
of $a^{\pm 1}$ in the expression for $v$.
The value of $m$ will be maximized if $s=+3$,
in which case, $m=(3)q^{n-1}$.
Finally, if $h \le n-2$, then
$m \le q^{n-2}(q-1)+q^{n-3}\lfloor {q \over 2} \rfloor+\cdots
 +(\lfloor {q \over 2} \rfloor)=
q^{n-1}-q^{n-2}+{{q^{n-2}-1} \over {2-1}}\lfloor {q \over 2} \rfloor$, so
$m<3q^{n-1}$.
\end{proof}

\begin{thm}\label{thm:BS1pnotmac}
The group $G_q=BS(1,q)=\langle a,t ~|~ tat^{-1}=a^q\rangle$
with $q \ge 7$ 
is not \mac\ with respect
to the generating set $A=\{a,a^{-1},t,t^{-1}\}$.
\end{thm}

\begin{proof}
Let $n$ be an arbitrary natural number with $n>100$.
Let $w':=t^{n}at^{-n}$ and
$u':=at^{n}at^{-(n-1)}$.  Then $w'$ and $u'$ are words
of length $R:=2n+1$.  
As a consequence of Lemma \ref{lem:boundmq}, an argument similar
to the proof of Lemma \ref{lem:wugeod} 
shows that the
words $w'$, $u'$, $w'a^{\pm 1}$, and $w't^{-1}$  are geodesics.
Hence $\ov{w'}$ and $\ov{u'}$ lie in $S(R)$.
Lemma \ref{lem:commute} says that
$w'at^{-1}=(a)(t^{n}at^{-n})t^{1}=_Gu'$, so
$d(\ov{w'},\ov{u'})=2$ and the word $\gamma:=at^{1}$ labels a
path from $\ov{w'}$ to $\ov{u'}$. 
Let $\delta$ be a path inside the ball $B(R)$ from $\ov{w'}$ to $\ov{u'}$
that has minimal possible length.  

Using the information on geodesics from the previous
paragraph, the
first letter of the path $\delta$ must be $t$. 
Let $\pi:\C \ra T$ denote the horizontal projection map
from the Cayley complex of $G_q$ to the regular tree $T$ of valence $q+1$.
The vertices $\pi(\ov{w'\delta(1)})=\pi(\ov{t})$ and 
$\pi(\ov{u'})=\pi(\ov{at})$ are
the terminal vertices of two distinct edges of $T$
with initial vertex $\pi(\ov{w'})=\pi(1)$.  Consequently,
there must be a
point $P:=\ov{w'\delta(j)}$ along the path $\delta$
with $\pi(P)=\pi(\ov{w'\delta(j)})=\pi(1)$
and $1 < j < l(\delta)$.
Write $\delta=\delta_1\delta_2$ where
$\delta_1$ is the subpath 
of $\delta$ from $\ov{w'}$ to $P$.

The vertex $P \ne \ov{w'}=\ov{a^{q^n}}$, and $P =_{G_q} a^m$
for some $m \in \Z$, so
Lemma \ref{lem:boundmq} shows that
$m \le 3q^{n-1}$.  The word $\delta_1$ labels a path from
$\ov{a^{q^{n}}}$ to $\ov{a^m}$, so
$\delta_1 =_{G_q} a^k$ with 
$k=q^{n}-m \ge (q-3)q^{n-1} >3q^{n-1}$ since $q \ge 7$.
Then Lemma \ref{lem:boundmq} says that either $k=q^n$ or
$\ov{a^k}$ is not in the ball $B(2n+1)$.  

If $\ov{a^k}$ is not in the ball $B(2n+1)$, then
the length $l(\delta_1) > 2n+1=R$.
Note that $u't^{-1}=aw'=_{G_q}a^{q^n+1}$.  The
word $\delta_2t^{-1}$ labels a path from 
$P=\ov{a^m}$ to $\ov{a^{q^n+1}}$, so
$\delta_2t^{-1} =_{G_q} a^{k+1}$ with 
$k+1>3q^{n-1}$.  Then the length $l(\delta_2t^{-1})>R$ as well.
Thus $l(\delta_2) \ge R$ and
the length $l(\delta) \ge R+1+R=2R+1$.  
Since there is a path $w'^{-1}u'$ of length $2R$ inside 
$B(R)$ from $\ov{w'}$ to $\ov{u'}$,
this contradicts our choice of $\delta$
with minimal length.

Then the path $\delta$ satisfies 
$k=q^n$, so $P=1$.  Therefore the path $\delta$ reaches
the vertex corresponding to the identity, and
the length of the path $\delta$ is $2R$.
\end{proof}

\begin{cor}
The properties \mac\ and \mpac\ are not commensurability invariant,
and hence also not quasi-isometry invariant.
\end{cor}

\begin{proof}
The index 3 subgroup of 
$BS(1,2)=\langle a,t ~|~ tat^{-1}=a^2 \rangle$
generated by $a$ and $t^3$ is isomorphic to $BS(1,8)$.
Theorem \ref{thm:BS12mac} shows that $BS(1,2)$ is \mpac\  and hence \mac,
and Theorem \ref{thm:BS1pnotmac} proves that $BS(1,8)$ has neither property.
\end{proof}

\section{Stallings' group is not \mac }

In \cite{\Stallings}, Stallings showed that the group
with finite presentation
\begin{eqnarray*}
S&:= & \langle a,b,c,d,s ~|~ [a,c]=[a,d]=[b,c]=[b,d]=1, \\
&& (a^{-1}b)^s=a^{-1}b, (a^{-1}c)^s=a^{-1}c, (a^{-1}d)^s=a^{-1}d\rangle
\end{eqnarray*}
does not have homological type $FP_3$.  In our notation,
$[a,c]:=aca^{-1}c^{-1}$ and $(a^{-1}b)^s:=sa^{-1}bs$.
Let $X:=\{a,b,c,d,s,a^{-1},b^{-1},c^{-1},d^{-1},s^{-1}\}$
be the 
inverse closed
 generating set, and let $\Gamma$ be the
corresponding Cayley graph of $S$.

Let $G$ be the subgroup of 
$S$ generated by
 $Y:=\{a,b,c,d,a^{-1},b^{-1},c^{-1},d^{-1}\}$, and
let $\Lambda$ be the corresponding Cayley graph of $G$.
Then $G$ is the
direct product of the nonabelian free groups
$\langle a,b \rangle$ and $\langle c,d\rangle$.
Let $H$ be the finitely generated subgroup of $G$ 
given by $H=\langle a^{-1}b, a^{-1}c, a^{-1}d\rangle$.

\comment{Is there a reference that proves Lemma \ref{lem:hexpsum} so
that we can skip the proof?}

\begin{lem}\label{lem:hexpsum}
The group $H$ consists of all elements of
$G$ that can be represented by a word 
over $Y$ of exponent sum zero.  Moreover,
every word,
over $X$ or $Y$, representing an element of $H$ must
have exponent sum zero.
\end{lem}

\begin{proof}
Using the fact that $a$ commutes with both $c$ and $d$,
we have that $ca^{-1}, da^{-1} \in H$, and so their inverses
$ac^{-1}, ad^{-1} \in H$.  
Using the fact that $b$ and $c$ commute,
\begin{eqnarray*}
ab^{-1}&=_S& a(c^{-1}c)b^{-1}=_S
(ac^{-1})b^{-1}c=_S (ac^{-1})b^{-1}(aa^{-1})c \\
&=_S& (ac^{-1})(a^{-1}b)^{-1}(a^{-1}c) \in H
\end{eqnarray*}
as well.
Taking products of the form 
$h_1^{-1}h_2$ with 
$$
h_1,h_2 \in \{a^{-1}b, a^{-1}c, a^{-1}d, ab^{-1}, ac^{-1}, ad^{-1}\}
$$
shows that $l_1^{-1}l_2, l_1l_2^{-1} \in H$ for all positive letters
$l_1,l_2 \in \{a,b,c,d\}$.
Finally, consider an arbitrary
word $w=l_1^{\epsilon_1} \cdots l_n^{\epsilon_n}$
with $l_i \in \{a,b,c,d\}$, $\epsilon_i=\pm 1$, and 
$\sum_i \epsilon_i=0$.  For each $i$, 
there is a letter $m \in Y$ which
commutes with both 
$l_i^{\epsilon_i}$ and $l_{i+1}^{\epsilon_{i+1}}$.  
Repeating the technique above 
of inserting
the inverse pair $mm^{-1}$ between
$l_i^{\epsilon_i}$ and $l_{i+1}^{\epsilon_{i+1}}$ 
and applying the commutation relations
as needed, then $w$ can be written as a product
of elements of exponent sum zero
of the form $m_1m_2$ with $m_i \in Y$.  Then $\ow \in H$.
The second sentence of this lemma follows from
the fact that the exponent sum for each of generators
of $H$ and each of the relators in
the presentation for $S$ is zero.
\end{proof}

Let $\phi:H \ra H$ be the identity function.
Then $S$ is the \hnn\ extension $S=G \star_{\phi}$
with stable letter $s$, and $s$ commutes with all
of the elements of $H$.

\begin{lem}\label{lem:geodg}
Let $w \in X^*$.  
\begin{enumerate}
\item If $w$ is a geodesic in $\Gamma$, then
the word $w$ cannot contain a subword of the form
$sus^{-1}$ or $s^{-1}us$ with $\ou \in H$.
\item If $\ow \in G$ and
$w$ is a geodesic in $\Gamma$, then $w \in Y^*$
and $w$ is a geodesic in $\Lambda$. 
\item If $w \in Y^*$ and $w$ is a geodesic in $\Lambda$, then
$w$, $sw$ and $ws$ are all geodesics in $\Gamma$.
\end{enumerate}
\end{lem}

\begin{proof}
Part (1) follows directly from the fact that
for $\ou \in H$, $sus^{-1}=_S s^{-1}us=_S u$.
In parts (2) and (3),  suppose that $g \in G$,
$v$ is a geodesic word over $X$ in $\Gamma$ 
representing $g$, and $w \in Y^*$ is a geodesic
in $\Lambda$ with $\ow=g$ also.
Then $vw^{-1}=_S 1$.  Britton's Lemma applied to the HNN extension
$S$ says that
if either $s$ or $s^{-1}$ occurs in
$vw^{-1}$, then $vw^{-1}$, and hence $v$, must contain a
subword of the form
$sus^{-1}$ or $s^{-1}us$ with $\ou \in H$, 
contradicting part (1).  Therefore $v \in Y^*$. 
Since  $v,w \in Y^*$ and
$v$ is is a geodesic in $\Lambda$, then $l(v) \le l(w)$
Similarly since $v,w \in X^*$ and $w$ is a geodesic in
$\Gamma$, $l(w) \le l(v)$.  Thus $v$ is also a geodesic
in $\Lambda$, and $w$ is a geodesic in $\Gamma$.
 
For the remainder of part (3), 
suppose that $\mu$ is a geodesic representative
of $sw$ in $\Gamma$.  Then $w^{-1}s^{-1}\mu=_S 1$.  
Britton's Lemma then says that $w^{-1}s^{-1}\mu$ must contain 
a subword of the form $sus^{-1}$ or $s^{-1}us$ with
$\ou \in H$.  Since $\mu$ is a geodesic, 
part (1) says that the $sus^{-1}$ or $s^{-1}us$
cannot be completely contained in $\mu$, so we can write
$\mu=\mu_1s\mu_2$ with
$\ov{\mu_1} \in H$.  Since $\mu_1$ and $s$ commute,
$s\mu_1\mu_2 =_S \mu =_S  sw$, so $\mu_1\mu_2 =_S w$ and both
$\mu_1\mu_2$ and $w$ are geodesics representing the same element of $G$.
Hence $l(\mu_1\mu_2) = l(w)$.  Then 
$l(\mu)=l(\mu_1)+1+l(\mu_2)=l(w)+1=l(sw)$, so $sw$ is a geodesic
in $\Gamma$.  The proof that $ws$ is also a geodesic
in $\Gamma$ is similar.
\end{proof}

The proof of the following theorem relies further 
on the HNN extension structure of Stallings' group $S$.  In
particular, we utilize an ``$s$-corridor'' to show that
the path $\delta$ in the definition of \mac\ cannot
exist.

\begin{thm}\label{thm:stall}
$(S,X)$ is not \mac\ with respect to the generating set $X$.
\end{thm}

\begin{proof}
Let $\alpha := b^{-(n+1)}a^{n+1}$ and $\beta := sb^{-(n+1)}a^n$, 
and let $\chi = b^{-(n+1)}a^n$ be their maximal common subword. 
The word $\alpha \in Y^*$, so $\ov{\alpha} \in G$;
in particular, the exponent sum of $\alpha$ is zero, so
Lemma \ref{lem:hexpsum} says
$\ov{\alpha} \in H$ also.
Since $\alpha$ is a geodesic in the Cayley graph $\Lambda$ of
the group $G=F_2 \times F_2$,  
Lemma \ref{lem:geodg}(3) says that 
$\alpha$ is a geodesic in $\Gamma$.
Similarly, $\chi$ is a geodesic in $\Lambda$, so
Lemma \ref{lem:geodg}(3) says that 
$\beta=s\chi$ is also geodesic in $\Gamma$.
Thus $\alpha $ and $\beta $ lie in the sphere of radius $2n+2$
in $\Gamma$.
Since 
$$
\alpha^{-1}\beta =
a^{-(n+1)}b^{n+1}sb^{-(n+1)}a^n=_S
sa^{-(n+1)}b^{n+1}b^{-(n+1)}a^n=_S sa^{-1},
$$ 
the distance $d(\ov{\alpha},\ov{\beta})=2$, for all
natural numbers $n$.

Suppose there is a path $\delta$ of length at 
most $2(2n+2)-1$ inside the
ball of radius $2n+2$ between $\ov{\alpha}$ and
$\ov{\beta}$.  
Since the relators in the presentation of $S$
have even length, the word $\delta$ must have length
at most $2(2n+2)-2=4n+2$.

Applying Britton's Lemma to the product
$\delta as^{-1}=_S1$ implies that 
$\delta=w_1sw_2$ with $\ov{w_1},\ov{w_2a} \in H$.
Then $\ov{w_1},\ov{w_2} \in G=F_2 \times F_2$. 
Lemma \ref{lem:geodg}(2)
and the direct product structure imply that
there are geodesic representatives 
$q_1$ and $q_2$ of $\ov{w_1}$ and $\ov{w_2}$,
respectively, that have the form
$q_1=q_{1_{a,b}}q_1{_{c,d}}$ and $q_2=q_{2_{c,d}}q_{2_{a,b}}$
with $q_{1_{a,b}}, q_{2_{a,b}} \in \{a,b,a^{-1},b^{-1}\}^*$
and $q_{1_{c,d}}, q_{2_{c,d}} \in \{c,d,c^{-1},d^{-1}\}^*$.

Since $\ov{\alpha}$ and $\ov{q_1}=\ov{w_1}$
are both elements of $H$,  
$\ov{\alpha q_1} \in H$ as well.
From the direct product structure, there is a
geodesic representative $\sigma \in Y^*$ of $\ov{\alpha q_1}$
of the form
$\sigma=\sigma_{a,b}\sigma_{c,d}$ 
with $\sigma_{a,b} \in \{a,b,a^{-1},b^{-1}\}^*$
and $\sigma_{c,d} \in \{c,d,c^{-1},d^{-1}\}^*$.
(see Figure \ref{fig:stall}).
\begin{figure}
\bt{ccc}\includegraphics[scale=.35]{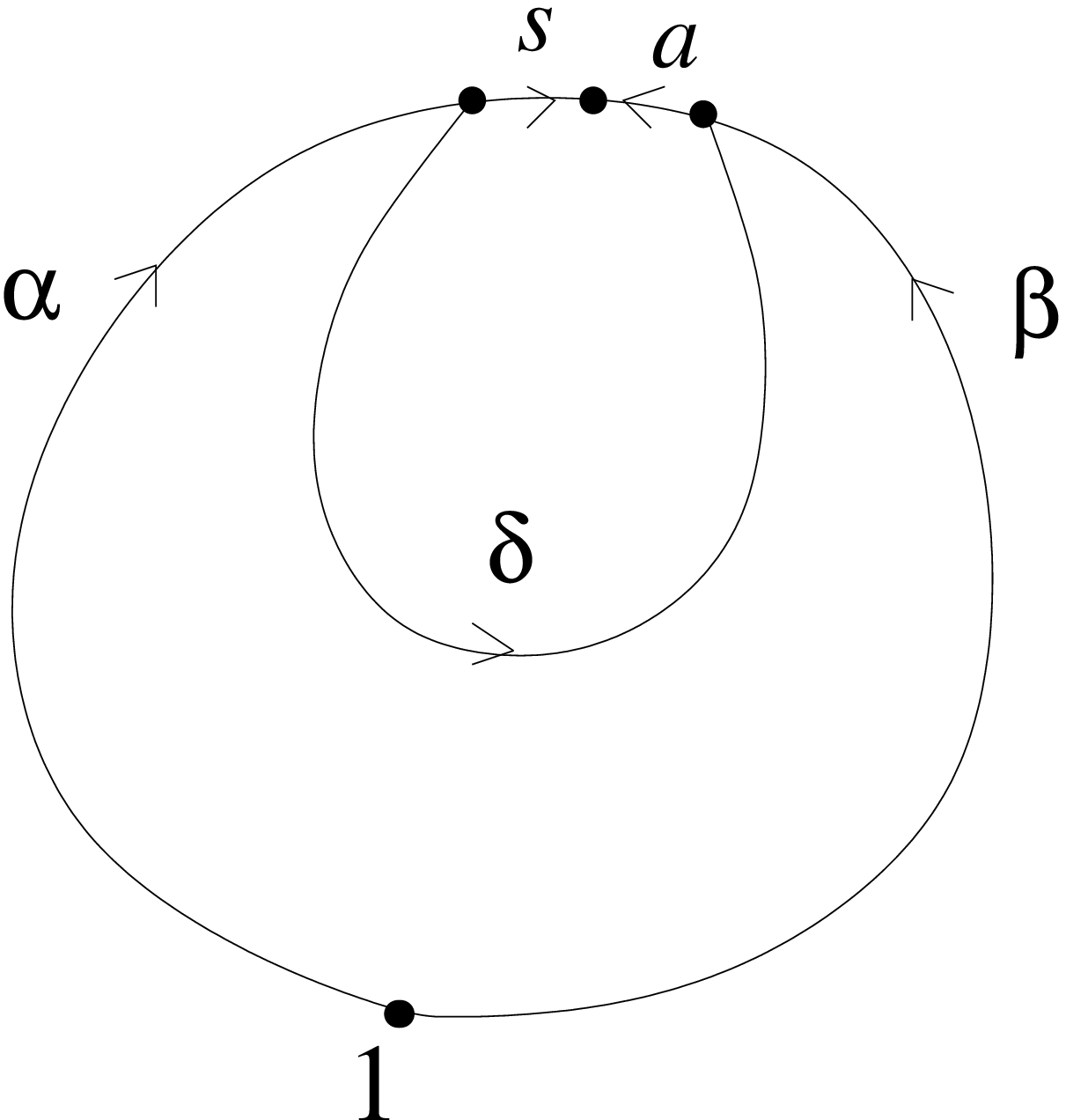}
& \hspace{5mm} &
\includegraphics[scale=.35]{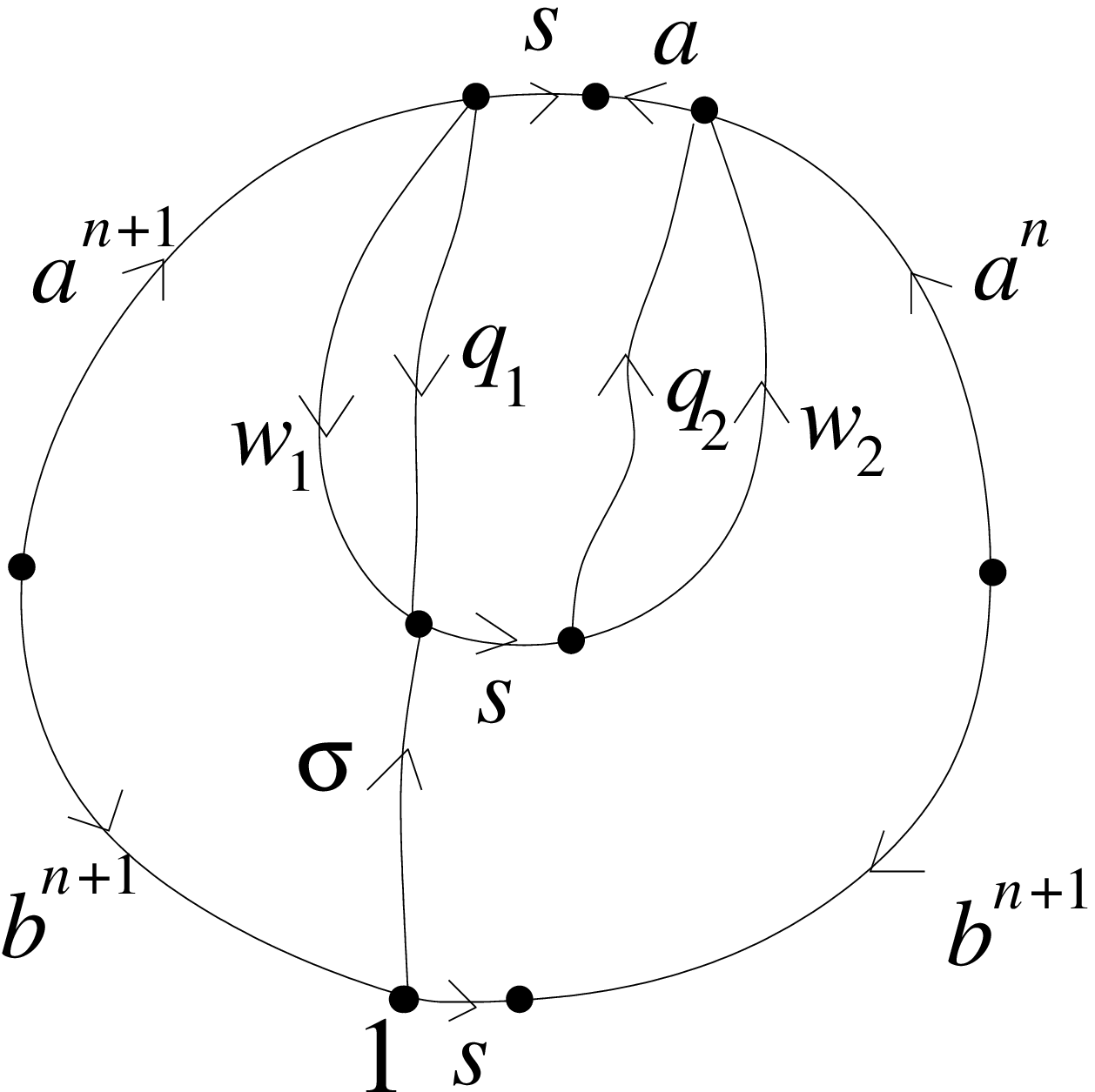}
\et
\caption{Paths in the Cayley graph of Stallings' group. \label{fig:stall}}
\end{figure}

The edge in $\Gamma$ labeled by $s$
connecting $\ov{\sigma}$ and
$\ov{\sigma s}$ is part of the path
$\delta$, so this edge must lie in the
ball of radius $2n+2$ in $\Gamma$.
Lemma \ref{lem:geodg}(3) says that $\sigma s$ is a geodesic,
so 
$d(1,\ov{\sigma})+1=l(\sigma)+1=l(\sigma s)=
d(1,\ov{\sigma s}) \le 2n+2$. 
Then the vertex
$\ov {\sigma} \in B(2n+1)$ and $l(\sigma) \le 2n+1$. 

Now $l(q_1)+1+l(q_2)\leq l(\delta)\leq 4n+2$, 
and thus 
either $l(q_1)\leq 2n$ or $l(q_2)\leq 2n$ (or both).

\bcase{\it A: If $l(q_1)\leq 2n$:}
Note that 
$\alpha q_1 \sigma^{-1} =b^{-(n+1)}a^{n+1} q_{1_{a,b}}q_1{_{c,d}}
\sigma_{c,d}^{-1}\sigma_{a,b}^{-1} =_{F_2 \times F_2} 1$.
Hence $q_1{_{c,d}}=_{F_2} \sigma_{c,d}$ and 
$\alpha q_{1_{a,b}}=_{F_2} \sigma_{a,b}$.  Since 
geodesics in free groups are unique, we also have
$q_1{_{c,d}}= \sigma_{c,d}$.  

There is an integer
$0 \le i_1 \le 2n$ such that $q_{1_{a,b}}(i_1)=\alpha^{-1}(i_1)$
but $q_{1_{a,b}}(i_1+1) \ne \alpha^{-1}(i_1+1)$, where
we denote $q_{1_{a,b}}(0):=1$ and $q_{1_{a,b}}(k):=q$
for all $k > l(q_{1_{a,b}})$.
Write $q_{1_{a,b}}=\alpha^{-1}(i_1)r$ with 
$r \in F_2=\langle a,b \rangle$.
The words $\alpha$, $q_{1_{a,b}}$, and $\sigma_{a,b}$
are all geodesic representatives
of elements of the free group $F_2$, 
and hence these are freely reduced words that define
non-backtracking edge paths in the tree given by
the Cayley graph for this group.  By definition of $i_1$, the 
product $\alpha q_{1_{a,b}}$ freely reduces
to $\alpha(2n+2-i_1)r$, with no further free reduction
possible.  Then $\alpha(2n+2-i_1)r$ is the unique geodesic
representative in $F_2=\langle a,b \rangle$
of $\alpha q_{1_{a,b}}$, and hence
$\alpha(2n+2-i_1)r = \sigma_{a,b}$.  

\bcase{\it A.1: If $i_1 \le n+1$:}
In this case, $q_{1_{a,b}}=a^{-i_1}r$.
Now $q_1=q_{1_{a,b}}q_{1_{c,d}}=a^{-i_1}rq_{1_{c,d}}$
represents an element of $H$, and so Lemma \ref{lem:hexpsum}
says $q_1$ has exponent sum zero.
Then $l(rq_{1_{c,d}}) \ge i_1$.  We also have
$\sigma = \sigma_{a,b} \sigma_{c,d}=
\alpha(2n+2-i_1)r q_{1_{c,d}}$.  Then the length
$l(\sigma) \ge (2n+2-i_1)+i_1=2n+2$, contradicting the
result above that $l(\sigma) \le 2n+1$.  Thus this
subcase cannot occur.

\bcase{\it A.2: If $i_1>n+1$:}
In this case, 
$\sigma_{a,b}=b^{-(2n+2-i_1)}r$.  Since
$\sigma=\sigma_{a,b}\sigma_{c,d}=b^{-(2n+2-i_1)}r\sigma_{c,d}$
represents an element of $H$, this word has exponent
sum zero, so $l(r\sigma_{c,d}) \ge 2n+2-i_1$ in this
subcase.  The word 
$q_1=q_{1_{a,b}}q_{1_{c,d}}=\alpha^{-1}(i_1)r\sigma_{c,d}$ then
has length $l(q_1) \ge i_1+(2n+2-i_1)=2n+2$, contradicting
the fact that we are in Case A.

\bcase{\it B: If $l(q_2)\leq 2n$:}
Since $\sigma \in H$, $\sigma$ commutes with $s$.  Then
$$
\sigma=_S s^{-1} \sigma s=_S s^{-1} \alpha q_1 s=_S
s^{-1} \alpha q_1 sq_2  q_2^{-1}=_S s^{-1} \beta q_2^{-1}
=_S \chi q_2^{-1}\ .
$$
In this case
$\sigma_{a,b}\sigma_{c,d}=_{F_2 \times F_2}
b^{-(n+1)}a^{n+1} q_{2_{a,b}}^{-1}q_2{_{c,d}}^{-1}$, so
$q_2{_{c,d}}^{-1}=_{F_2} \sigma_{c,d}$ and 
$\chi q_{2_{a,b}}^{-1}=_{F_2} \sigma_{a,b}$.  
Uniqueness of geodesics in $F_2=\langle c,d \rangle$ implies
$q_2{_{c,d}}^{-1}= \sigma_{c,d}$.  

There is an integer
$0 \le i_2 \le 2n$ such that 
$q_{2_{a,b}}^{-1}(i_2)=\chi^{-1}(i_2)$
but $q_{2_{a,b}}^{-1}(i_2+1) \ne \chi^{-1}(i_2+1)$.
Write $q_{2_{a,b}}=r(\chi^{-1}(i_2))^{-1}$ with 
$r \in F_2=\langle a,b \rangle$.
The words $\chi$, $q_{2_{a,b}}$, and $\sigma_{a,b}$
are all geodesics, and hence freely reduced words,
in $F_2$.  By definition of $i_2$, the 
product $\chi q_{2_{a,b}}^{-1}$ freely reduces
to $\chi(2n+1-i_2)r^{-1}$, with no further reduction
possible.  Then 
$\chi(2n+1-i_2)r^{-1} = \sigma_{a,b}$.  

\bcase{\it B.1: If $i_2 \le n$:}
In this case, $q_{2_{a,b}}=ra^{i_2}$.
Now $q_2=q_{2_{c,d}}q_{2_{a,b}}=q_{2_{c,d}}ra^{i_2}$.
Recall that $q_2$ was chosen
as a geodesic representative of an
element $\ov{w_2} \in G$ for which
$\ov{w_2a} \in H$.  Then $q_2a$ represents
an element of $H$, and so (by Lemma \ref{lem:hexpsum})
has exponent sum zero.
Therefore the exponent sum of $q_2$ is -1.
Then $l(q_{2_{c,d}}r) \ge i_2+1$.  We also have
$\sigma = \sigma_{a,b} \sigma_{c,d}=
\chi(2n+1-i_2)r^{-1} q_{2_{c,d}}^{-1}$.  Then the length
$l(\sigma) \ge (2n+1-i_2)+(i_2+1)=2n+2$, again
contradicting the
result above that $l(\sigma) \le 2n+1$.

\bcase{\it B.2: If $i_2>n$:}
In this case, 
$\sigma_{a,b}=b^{-(2n+1-i_2)}r^{-1}$.  Since
$\sigma=\sigma_{a,b}\sigma_{c,d}=b^{-(2n+1-i_2)}r^{-1}\sigma_{c,d}$
represents an element of $H$, this word has exponent
sum zero, so $l(r^{-1}\sigma_{c,d}) \ge 2n+1-i_2$ in this
subcase.  Therefore the word 
$q_2=q_{2_{c,d}}q_{2_{a,b}}=\sigma_{c,d}^{-1}r(\chi^{-1}(i_1))^{-1}$ 
has length $l(q_2) \ge (2n+1-i_2)+i_2=2n+1$, contradicting
the fact that we are in Case B. 

\vspace{3mm}

Therefore every possible subcase results in a
contradiction implying that the subcase cannot occur.
Then the path $\delta$ cannot exist, so $S$ is not
\mac\ with respect to the generating set $X$. 
\end{proof}

\bibliography{refs_mac}
\bibliographystyle{plain}

\end{document}